\documentclass[12pt]{article}
\usepackage{amsmath, amssymb, amsthm}
\usepackage{geometry}
\usepackage{graphicx}
\usepackage{hyperref}
\usepackage{import}
\usepackage{tikz-cd}
\usepackage{float}
\usepackage{placeins}
\hypersetup{colorlinks, citecolor=blue}
\usepackage{tabularx}
\usepackage{booktabs}
\usepackage{subcaption}
\usepackage{authblk}
\usepackage{natbib}

\newtheorem{theorem}{Theorem}[section]
\newtheorem{lemma}[theorem]{Lemma}
\newtheorem{proposition}[theorem]{Proposition}
\newtheorem{corollary}[theorem]{Corollary}
\theoremstyle{definition}
\newtheorem{definition}[theorem]{Definition}

\theoremstyle{remark}

\newcommand{\RR}{\mathbb R}

\newcommand{\abs}[1]{\left\vert#1\right\vert} 
\newcommand{\set}[2]{\left \{ \left. #1 \,\right| #2 \right\}  } 
\newcommand{\arrowmap}[5]{
    #1 \colon #2 &\to #3 \notag \\ 
    #4 &\mapsto #5
}

\newcommand{\norm}[1]{\left\lVert#1\right\rVert} 
\newcommand{\inprod}[2]{\left\langle #1, #2\right\rangle} 

\providecommand{\keywords}[1]
{
  \small	
  \textbf{Keywords:} #1
}

\title{A Riemannian covariance for manifold-valued data}
\author[1]{Meshal Abuqrais}
\author[2]{Davide Pigoli}
\affil[1,2]{King's College London}
\date{}

\begin{document}

\maketitle
\begin{abstract}
The extension of bivariate measures of dependence to non-Euclidean spaces is a challenging problem.
The non-linear nature of these spaces makes the generalisation of classical measures of linear dependence (such as the covariance) not trivial.
In this paper, we propose a novel approach to measure stochastic dependence between two random variables taking values in a Riemannian manifold, with the aim of both generalising the classical concepts of covariance and correlation and building a connection to Fréchet moments of random variables on manifolds.
We introduce generalised local measures of covariance and correlation and we show that the latter is a natural extension of Pearson correlation.
We then propose suitable estimators for these quantities and we prove strong consistency results.
Finally, we demonstrate their effectiveness through simulated examples and a real-world application.
\end{abstract}

\keywords{Geometric statistics, Object data analysis, Non-Euclidean data, Stochastic dependence, Fréchet moments, Vectorcardiogram data.}

\section{Introduction}

The statistical analysis of data belonging to non-Euclidean spaces has attracted significant attention in recent years (see, e.g., \cite{marron2021object, patrangenaru2016nonparametric}). The practical importance of statistical analysis for non-Euclidean data stems from the need to handle complex and diverse data structures such as shape data in medical imaging (\cite{bharath2018radiologic}), network data in linguistics (\cite{severn2022manifold}), and probability density functions in environmental sciences (\cite{menafoglio2021object}), to name just a few. 
Since classical statistical techniques were developed for Euclidean (flat) spaces, they are often unsuitable when the space of interest exhibits a more complex geometry, such as nonzero curvature.

Historically, one can argue that the development of statistical methodology for non-Euclidean data started with Fr\'{e}chet, with his work on mean points in general metric spaces (\cite{frechet1948elements}). 
Subsequently, several branches of statistics addressed non-Euclidean data, including directional statistics (\cite{mardia2009directional}), statistical shape analysis (\cite{dryden2016statistical}), and compositional data (\cite{aitchison1982statistical}). 
More recently, the common framework of object data analysis (\cite{marron2014overview}) has been developed to analyse these kinds of data, and an extensive toolbox is now available, in particular for the case of data taking values in manifolds or metric spaces (\cite{patrangenaru2016nonparametric}). 
Statistical methods for manifold-valued data have been used, for example, in medical statistics (see, e.g., \cite{pennec-sommer-fletcher-2019riemannian-medical}) and medical imaging (e.g., diffusion tensor data \cite{fletcher2007riemannian}).
However, the question of how to model and measure dependence between manifold-valued random variable is still an open research question. This is crucial for example to extend existing methodology to non independent sample, e.g. in the case of time series of manifold-valued data.  

Manifolds, in general, are not vector spaces, which means that traditional notions of dependence, such as covariance and linear correlation (e.g., Pearson correlation), do not directly extend to Riemannian manifolds. 
This limitation has motivated efforts to generalise these concepts beyond Euclidean settings. 
Some of these efforts have focused on extending to non-linear spaces measures that were originally developed to capture non-linear dependence in Euclidean spaces. 
For example, \cite{distance-covariance-in-metric-spaces-Lyons} successfully extended the foundational work of \cite{szekely2007measuring} on distance covariance and distance correlation. 
Specifically, \cite{distance-covariance-in-metric-spaces-Lyons} showed that if the random variables take values in a metric space of strong negative type, then the distance correlation satisfies the zero-correlation-independence criterion (i.e., the correlation between random variables is zero if and only if they are independent). 
Moreover, it was shown in the same paper that the assumption that the metric space is of strong negative type is not only sufficient but also necessary for the zero-correlation-independence equivalence to hold for distance correlation. 
Unfortunately, many metric spaces encountered in practice are not of negative type. For example, \cite{hjorth2002hyperbolic} demonstrated that any compact Riemannian manifold that is not simply connected cannot be of negative type as a metric space. 
This implies that the torus $\mathbb{T}^2$ or SO(3) are not of negative type. 
Moreover, the distance correlation measures the strength of the dependence but it doesn't distinguish between linear and non-linear dependence and it doesn't give any indication of the direction of this dependence, being a positive measure. 
Similar drawbacks afflict for example the Ball covariance proposed by \cite{pan2020ball}.
Similar ideas have been previously applied in specific settings. Recently, \cite{shao2022intrinsic} introduced a measure of covariance for Riemannian functional data by mapping random variables to the tangent spaces at their respective Fréchet means and defining the covariance as the covariance operator for the resulting random tangent vectors. Although this approach resembles our proposed measure, it does not account for the intermediate geometry between the means.
Other measures of dependence have been proposed for specific manifolds such as the torus and sphere in the context of spherical regression; see \cite{zhan2019circular} and \cite{downs2003spherical}. 
These measures rely heavily on the underlying geometry of the space and therefore cannot be easily extended to other Riemannian manifolds. 

The aim of this work is to introduce a novel measure of dependence, the Riemannian covariance, which offers a clear geometric interpretation and applies to a broad class of manifolds. The Riemannian covariance can be seen as both a generalisation of the classical covariance in Euclidean spaces and a measure compatible with the concept of Fréchet moments (\cite{patrangenaru2016nonparametric}), the commonly used framework for describing the moments of manifold-valued random variables. 
While the Riemannian covariance (and the corresponding Riemannian correlation) can be defined under more general conditions, this work focuses on the case where the underlying space is a compact and connected Riemannian manifold. 
Additionally, under certain assumptions on the distribution support, we prove the strong consistency of estimators for both the Riemannian covariance and Riemannian correlation. 
To test the robustness of our generalised covariance and correlation measures, we will compare them with existing measures of dependence based on distance covariance (\cite{distance-covariance-in-metric-spaces-Lyons}) by simulations on $\mathbb{S}^2$ and SO(3).

The paper is organised as follows. 
Section 2 provides the necessary background on differential geometry and probability theory on Riemannian manifolds. 
In this section, we also describe the Fréchet function, a fundamental concept upon which our work is built. 
In Section 3, we introduce the proposed Riemannian covariance and Riemannian correlation, highlighting some of their essential properties. Additionally, we provide natural sample estimators for these measures. 
Section 4 presents the main results of this work and explores their implications. 
Specifically, we derive a strong consistency theorem for the sample generalised covariance/correlation on Riemannian manifolds. 
Section 5 includes several simulation studies designed to assess the finite sample properties of our proposed estimators in two examples of Riemannian manifolds. 
Furthermore, we apply these measures to real-world datasets, specifically vectorcardiogram data, showcasing their applicability and effectiveness in this context. Finally, Section 6 presents conclusions and some future research directions.

\section{An overview of differential geometry and probability theory on manifolds}
In this section, we provide some essential background of differential geometry and probability theory on manifolds.
Our discussion on the geometric aspects is primarily based on \cite{Rie-Lee} and \cite{tu2017differential}. 
For probability theory on manifolds and metric spaces, we refer the readers to \cite{pennec2006intrinsic} and \cite{patrangenaru2016nonparametric}.
\subsection{Riemannian manifolds}
An $n$-dimensional Riemannian manifold is a pair $(M, \inprod{\cdot}{\cdot})$ consisting of a smooth $n$-dimensional manifold $M$ together with a smooth symmetric covariant 2-tensor field $\inprod{\cdot}{\cdot}$ that is positive definite at each point of $M$. This tensor field is called the Riemannian metric of $M$. 
We shall denote the Riemannian metric at a point $p \in M$ by $\inprod{\cdot}{\cdot}_p$. 
For each $p\in M$, the Riemannian metric gives an inner product on the tangent space $T_pM$, given by the map $(v, w) \mapsto \inprod{v}{w}_p$ for all $v, w \in T_pM$.
The norm induced by the inner product is denoted by $\norm{\cdot}_p$ or $\norm{\cdot}$ if the point $p$ is understood.
From now on, we assume $M$ is an $n$-dimensional compact and connected Riemannian manifold.

Let $C^\infty(M)$ denote the space of all smooth real-valued functions on the manifold $M$. The vector space $\mathfrak{X}(M)$ consists of all smooth vector fields on $M$, with operations defined pointwise.

A connection on $M$ is a map
\begin{equation*}
\nabla:\mathfrak{X}(M)\times \mathfrak{X}(M) \longrightarrow \mathfrak{X}(M),
\end{equation*}
written as $\nabla_X Y$ instead of $\nabla(X,Y)$, that is $C^\infty(M)$-linear in $X$ and satisfies the Leibniz product rule 
\begin{equation*}
\nabla_X(fY) = (Xf)Y + f\nabla_X Y,
\end{equation*}
for all $f \in C^\infty(M)$ and $X, Y \in \mathfrak{X}(M)$.

The connection is called symmetric if 
\begin{equation*}
\left[X,Y\right] = \nabla_{X}Y - \nabla_Y X,
\end{equation*}
for all $X, Y \in \mathfrak{X}(M)$. 
Additionally, $\nabla$ is said to be compatible with the metric of $M$ if, for all $X, Y, Z \in \mathfrak{X}(M)$,
\begin{equation*}
Z\inprod{X}{Y} = \inprod{\nabla_Z X}{Y} + \inprod{X}{\nabla_Z Y}.
\end{equation*}

One of the fundamental results in Riemannian geometry is that for a Riemannian manifold $(M,\inprod{\cdot}{\cdot})$, there exists a unique symmetric connection on $M$ that is compatible with the metric $\inprod{\cdot}{\cdot}$.
This connection is called the Levi-Civita connection on $M$.
In this article, we assume that all Riemannian manifolds are equipped with their Levi-Civita connection.
For any smooth curve $\gamma\colon I\to M$, let $\mathfrak{X}(\gamma)$ be the vector space of all smooth vector fields along $\gamma$.
Then the connection $\nabla$ induces a unique operator $D_t\colon \mathfrak{X}(\gamma)\to\mathfrak{X}(\gamma)$ called the covariant derivative along $\gamma$. $D_{t}$ is $\mathbb{R}$-linear, satisfies the product rule, and for any smooth extension $\widetilde{V}$ of $V\in\mathfrak{X}(\gamma)$,
\begin{equation*}
D_{t}V(t) = \nabla_{\gamma'(t)}\widetilde{V}.
\end{equation*}
A curve $\gamma$ in $M$ is called a geodesic if
\begin{equation*}
D_{t}\gamma' = 0.
\end{equation*}
In the standard Euclidean space $\mathbb{R}^n$ the covariant differentiation of vector field along a curve corresponds to the acceleration vector field of curve.
So geodesics in $\mathbb{R}^n$ are precisely straight lines.
In this sense, geodesics are generalisation of straight lines in Euclidean spaces.
The tangent bundle of $M$ is defined as
\begin{equation*}
TM=\set{(p,v)}{p\in M, v\in T_{p}M}.
\end{equation*}
A fundamental result in the theory of smooth manifolds states that for any $(p, v) \in TM$, there exists a unique maximal geodesic $\gamma_v(t; p)$ defined on an interval containing $0$, such that $\gamma_v(0; p) = p$ and $\gamma'_v(0; p) = v$. If the point $p$ is understood from the context, we often write $\gamma_v(t)$ instead of $\gamma_v(t; p)$.

Since we are assuming $M$ is compact, the Hopf-Rinow theorem implies that for all $p \in M$ and $v \in T_pM$, the geodesic $\gamma_v(t; p)$ is defined on the entire real line $\mathbb{R}$.
The exponential map is then defined as
\begin{align*}
\arrowmap{\exp}{TM}{M}{(p,v)}{\exp_p(v) = \gamma_v(1; p)}.
\end{align*}
The definition of the exponential map implies that for all $v \in T_pM$, the map $t \mapsto \exp_p(tv)$ is a geodesic curve parameterised by $t \in \mathbb{R}$.

The exponential map is smooth from $TM$ to $M$.
Therefore, by the inverse function theorem, for each $p \in M$, there exists a neighborhood $W$ of $0 \in T_pM$ such that $\exp_{p}\colon W \to \exp_{p}(W)$ is a diffeomorphism.
In particular, there exists $\varepsilon > 0$ such that $\exp_{p}\colon B_{\varepsilon}(0) \to \exp_p(B_{\varepsilon}(0))$ is a diffeomorphism.
Therefore, the exponential map is a local diffeomorphism at $p$, meaning it can serve as a coordinate map once $T_pM$ is identified with $\mathbb{R}^n$.
The inverse map of $\exp_p(\cdot)$ is called the logarithm map at $p$, denoted by $\log_p(\cdot)$. 
Under this identification, we can perform concrete computations with the logarithm map as a chart from a subset of $M$ to $\mathbb{R}^n$.
If $\exp_p\colon W \to M$ is a diffeomorphism, its image is called the normal neighborhood of $p$. When $W = B_0(\varepsilon)$ is an open ball centred at $0$ with radius $\varepsilon > 0$, the set $\exp_p(W)$ is called a geodesic ball centred at $p$, which we denote by $\mathcal{B}_p(\varepsilon)$.
The largest $r > 0$ such that $\exp_p(B_r(0))$ is a geodesic ball is called the injectivity radius at $p$, written as $\operatorname{inj}(p)$.
The injectivity radius of the entire manifold $M$ is defined as $\operatorname{inj}(M) = \inf\limits_{p \in M} \operatorname{inj}(p)$.

We now consider $M$ as a metric space. 
Let $\mathcal{L}_{p}^{q}$ be the set of all piecewise smooth curves $\gamma$ on $[0,1]$ such that $\gamma(0) = p$ and $\gamma(1) = q$.
The length of $\gamma \in \mathcal{L}_{p}^{q}$ is defined as
\begin{equation*}
L(\gamma) = \int_{0}^1 \|\gamma'(t)\|_{\gamma(t)} \, dt.
\end{equation*}
For all $p, q \in M$, the distance between $p$ and $q$ is defined as
\begin{equation*}
d(p, q) = \inf_{\gamma \in \mathcal{L}_{p}^q} \{L(\gamma)\}.
\end{equation*}
When dealing with $M$ as a metric space, we are assuming the distance function above.
A curve $\gamma$ in $\mathcal{L}_{p}^q$ is said to be length-minimising or minimising if $L(\gamma) = d(p, q)$.
Given normal coordinates on a geodesic ball \(\mathcal{B}_{p}\), then
for all \(q\in\mathcal{B}_{p}\),
\begin{equation*}
d(p,q) = \norm{\log_{p}q},
\end{equation*}
where $\norm{\cdot}$ here is the Euclidean norm in the tangent space $T_pM$, which is identified with $\mathbb{R}^n$ in these coordinates.
Because of this property, $\log_{p}(\cdot)$, or $\exp_{p}(\cdot)$, is called a local radial isometry at $p$.
A geodesic ball $\mathcal{B}_p$ is called a geodesically convex ball if for all $x, y \in \mathcal{B}_p$, there exists a unique length-minimising geodesic connecting $x$ and $y$. 
Geodesically convex balls are denoted by $\mathcal{CB}$.
The largest $r > 0$ such that $\exp_p(B_0(r))$ is a convex ball is called the convexity radius of $M$ at $p$.
The convexity radius of $M$ is defined as $\operatorname{conv}(M) = \inf\limits_{p \in M} \operatorname{conv}(p)$. 
A basic theorem in Riemannian geometry states that for a compact manifold $M$, $0 < \operatorname{conv}(M) \leq \frac{1}{2} \operatorname{inj}(M)$, see \cite{berger2003panoramic}.

Next, we set up the integration framework on Riemannian manifolds.
Let $(U, x)$ be a smooth chart on $M$.
Define $G$ as the matrix with components given by $G_{ij} = \left\langle \frac{\partial}{\partial x^i}, \frac{\partial}{\partial x^j} \right\rangle$, where $i, j = 1, \ldots, n$.
The Riemannian volume form on $M$ is defined as
\begin{equation*}
dV = \sqrt{\det(G)} \, dx^1 \cdots dx^n.
\end{equation*}
The volume form $dV$ induces a measure on the Borel $\sigma$-algebra of $M$. 
This volume form is also referred to in the literature as the "Riemannian measure" or "volume measure." 
Therefore, it is possible to define integrals of real-valued measurable functions on $M$.
However, our focus will be on the integration of measurable functions with compact support.

Given a measurable function $f: M \to \mathbb{R}$ with compact support contained within the chart $U$, the integral of $f$ with respect to the volume form $dV$ is defined by
\begin{equation*}
\int_{U} f \, dV = \int_{x(U)} \left(f \circ x^{-1}\right) \sqrt{\det(G \circ x^{-1})} \, dx^1 \cdots dx^n.
\end{equation*}

More generally, let $f: M \to \mathbb{R}$ be a measurable function with compact support.
Suppose $\{(U_\alpha, x_\alpha)\}_{\alpha \in \mathcal{A}}$, where $\mathcal{A}$ is an indexing set, is a cover of $M$, and let $\{\rho_\alpha\}_{\alpha \in \mathcal{A}}$ be a smooth partition of unity subordinate to this cover. 
Then the integral of $f$ with respect to $dV$ is defined as
\begin{equation*}
\int_{M} f \, dV = \sum_{\alpha \in \mathcal{A}} \int_{U_\alpha} \rho_\alpha f \, dV.
\end{equation*}
It can be shown that this definition is independent of the choice of charts, and, thus, the integral is well-defined.
The volume of a Borel subset $U \subset M$ is defined by
\begin{equation*}
\text{Vol}(U) = \int_{U} dV.
\end{equation*}
Because we are interested in compact manifolds, the volume is always finite.

\subsection{Probability theory on Riemannian manifolds}
Let $(\Omega, \Sigma, P)$ be a probability space, and let $M$ be a compact and connected $n$-dimensional Riemannian manifold. 
The random variables of interest are measurable functions from $\Omega$ to $M$. 
We refer to such functions as $M$-valued random variables or random objects. 
Similar to density functions in Euclidean spaces that are absolutely continuous with respect to the Lebesgue measure, we consider densities that are absolutely continuous with respect to the volume measure on Riemannian manifolds. 
If $X$ is an $M$-valued random variable, its probability distribution is the probability measure on $\mathcal{B}(M)$ given by $P_X(B) = P \circ X^{-1}(B)$ for all $B \in \mathcal{B}(M)$. 
All $M$-valued random variables are assumed to have probability distributions that are absolutely continuous with respect to the volume measure on $M$.

One of the major obstacles in performing data analysis for manifold-valued random variables is the nonlinear nature of the space in which the data lie. 
Furthermore, there are no natural coordinates or points of reference like point of origin, unlike linear spaces such as $\mathbb{R}^n$.
Therefore, careful thought has to be given when defining basic quantities such as mean or variance.
The problem of handling random variables that assume values in a metric space was first considered by \cite{frechet1948elements}.
He observed that the mean $\mu$ of a random vector $X: \Omega \to \mathbb{R}^n$ could be characterised as the unique minimizer of the functional
\begin{equation*}
\mu = \text{arg}\min\limits_{p \in \mathbb{R}^n} E\left[\|X - p\|^2\right] = \text{arg}\min\limits_{p \in \mathbb{R}^n}  E\left[d(X, p)^2\right],
\end{equation*}
where $d$ is the Euclidean distance function on $\mathbb{R}^n$.
This serves as the starting point for defining moments for random variables taking values in a manifold or metric space.
While one could define the Fréchet function for random objects in metric spaces, we restrict ourselves here to compact and connected Riemannian manifolds.

Formally, if $X: \Omega \to M$ is a random variable whose probability measure is $Q_X$, the $r$-th Fréchet function of $X$ is defined by
\begin{equation*}
\mathcal{F}_{r,X}(p) = E\left[d^{r}(X, p)\right] = \int_{M} d^{r}(p, x) \, dQ_X(x). 
\end{equation*}
For $\mathcal{F}_{2,X}(p)$, we define $\sigma_X^2 = \inf_{p \in M} \mathcal{F}_{2,X}(p)$ to be the Fréchet total variance of $X$.
The set of minimizers of $\mathcal{F}_{2,X}(p)$ is called the Fréchet mean set. 
Because we are interested in the case when $r = 2$, we write $\mathcal{F}_X(p)$ instead of $\mathcal{F}_{2,X}(p)$ or $\mathcal{F}_X$ when there is no confusion. The existence of minimizers is not guaranteed in the general case when the random object takes values in a metric space. 
However, in our setting, the compactness of the manifold implies the existence of extreme values for the $r$-th Fréchet function of $X$, see Lemma \ref{lemma: boundedness of Frechet function}.

An immediate observation is that the Fréchet mean may not be a singleton. In other words, the random object $X$ may have more than one mean. Clearly, this phenomenon does not occur in Euclidean space $\mathbb{R}^n$. For example,  an $\mathbb{S}^2$-valued random variable with uniform distribution has the entire space $\mathbb{S}^2$ as its Fréchet mean set. One existence and uniqueness result for the minimizers of the Fréchet second moment function, $r = 2$, in the context of Riemannian manifolds can be found in Chapter 8 of \cite{buser1981karcher-gromov}. In particular, it states that if the $M$-valued random variable $X$ is supported in a geodesically convex ball $\mathcal{B}$, then $X$ has a unique Fréchet mean $\mu \in \mathcal{B}$. Furthermore, the Fréchet mean $\mu$ is characterised as the unique $\mu$ for which
\begin{equation*}
E[\log_{\mu}X] = 0.
\end{equation*}
Moreover, global uniqueness holds for special manifolds such as Hadamard manifolds; see \cite{pennec2006intrinsic}.
Sufficient conditions for the existence and uniqueness of minimizers for the $r$-th Fréchet function were established in \cite{afsari2011riemannian} based on the injectivity radius and sectional curvature of the underlying manifold.

Next, we introduce the empirical version of the Fréchet mean set.
Let $X_1, \ldots, X_n$ be independent $M$-valued random variables with a common probability measure $Q$. 
Let
\begin{equation*}
\widehat{Q}_n = \frac{1}{n} \sum_{k=1}^n \delta_{X_k},
\end{equation*}
be their empirical distribution. The Fréchet sample mean set is the set of minimizers $\widehat{\mu}_n$ of the Fréchet function associated with $\widehat{Q}_n$. Specifically, the Fréchet sample mean $\widehat{\mu}_n$ is given by
\begin{equation}\label{Frechet sample mean}
\widehat{\mu}_n=\text{arg}\min\limits_{p \in {M}}\left(\frac{1}{n} \sum_{k=1}^n d^2(X_k, p)\right),
\end{equation}
A general estimation result for the Fréchet mean set is the strong consistency, or almost sure convergence, of the Fréchet sample mean set $\widehat{\mu}_n$, as given in (\ref{Frechet sample mean}), to the Fréchet mean set of $X$. See \cite{patrangenaru2016nonparametric}, Theorem 4.2.4.
As with the Fréchet mean set, the sample mean set may not be a singleton.
However, the existence and uniqueness results discussed above for the Fréchet mean also apply to the Fréchet sample mean.

\section{Riemannian covariance and Riemannian correlation}

Having established the geometric and probabilistic frameworks for analysing manifold-valued data, we are now in a position to introduce the central concepts of this paper: Riemannian covariance and Riemannian correlation. These measures, formulated as functions, extend the familiar notions of covariance and correlation from classical Euclidean spaces to Riemannian manifolds.
We will begin by formally defining these dependence measures and exploring their fundamental properties and statistical interpretations. 
Then, we will introduce the corresponding empirical estimators for each measure.

Throughout this section, we assume that $M$ is a compact and connected Riemannian $n$-manifold. We denote the geodesic ball centred at $p$ by $\mathcal{B}_{p}$. For any $M$-valued random variable, we assume that its density function is absolutely continuous with respect to the Riemannian volume measure of $M$.
 
\subsection{Definition of Riemannian covariance}\label{Covariance in convex geodesic balls section}
Our starting point is a note made in \cite{pennec2006intrinsic} which gave an analogue for cross-covariance of random variables in Riemannian manifolds.
In particular, let $X\colon\Omega\to M$ be a random object with unique Fréchet mean $\mu$. 
If $X$ is supported in normal neighborhood of $\mu$, then the cross-covariance matrix of $X$ is defined as
\begin{equation}\label{pennec-covariance}
\Sigma_{\mu}(X,X)=E\left[(\log_{\mu}X)(\log_{\mu}X)^T\right],
\end{equation}
This definition is naturally dependent on the point $\mu$. While it serves as an appropriate measure for a single random object or for two random objects sharing the same Fréchet mean, it becomes less clear how to extend this measure to cases where the random objects have different Fréchet means.

To address this, we extend the dependence on $\mu$ by considering it as a function on $M$. However, since the domain of the logarithm map (the inverse of the exponential map) is not generally the entire manifold, it is necessary to first specify the correct domain for such measure.

\begin{definition}
Let $X$ and $Y$ be both $M$-valued random variables with densities $f_X$ and $f_Y$, respectively. We define the domain of comparison between $X$ and $Y$ as the set
\begin{equation*}
\mathcal{D}(X,Y)=\set{p\in M}{\text{supp$f_X$},\text{supp$f_Y$}\subset\mathcal{B}_{p}\text{ a.s.}}.
\end{equation*}

If $X=Y$ almost surely, then we write $D(X)$ for $\mathcal{D}(X,X)$.
\end{definition}

Motivated by (\ref{pennec-covariance}), we take one step further by introducing the following concept of covariance between two random objects.
\begin{definition} [Riemannian covariance]
Suppose that $X$ and $Y$ are $M$-valued and assume $p\in\mathcal{D}(X,Y)$.
Let the Riemannian cross-covariance matrix of $X$ and $Y$ at $p$ be
\begin{equation*}
\Sigma_{p}(X,Y)=E\left[\log_{p}X\log_{p}Y^T\right]-E\left[\log_{p}X\right]E\left[\log_{p}Y\right]^T.
\end{equation*}
The Riemannian covariance between $X$ and $Y$ at $p$ is then defined as
\begin{equation*}
\text{Rcov}_{p}(X,Y)=tr(\Sigma_{p}(X,Y)).
\end{equation*}
\end{definition}
In other words, we have $\Sigma_p(X,Y)$ and Rcov$_p(X,Y)$
\begin{align*}
\arrowmap{\Sigma(X,Y)}{\mathcal{D}(X,Y)}{\RR^{n\times n}}{p}{\Sigma_{p}(X,Y)},
\end{align*}
and
\begin{align*}
\arrowmap{\text{Rcov}(X,Y)}{\mathcal{D}(X,Y)}{\RR}{p}{\text{Rcov}_{p}(X,Y)},
\end{align*}
as locally defined Euclidean-valued functions on $\mathcal{D}(X,Y)\subset M$.
Note that the cross-covariance matrix is not basis independent.
However, the Riemannian covariance at $p$ is in fact independent of the basis choice since it involves the trace.
Therefore, it is more convenient to deal with Rcov instead of $\Sigma$.
An immediate result is the following.
\begin{proposition}\label{proposition Rcov(X,X) and Frechet function simple relation}
Let $M$ be a compact and connected Riemannian manifold. Assume that $X$ is an $M$-valued random variable and $p\in\mathcal{D}(X)$. Then
\begin{equation*}
\text{Rcov}_{p}(X,X)=tr(\Sigma_{p}(X,X))=\mathcal{F}_{2}(p)-\norm{E\left[\log_{p}X\right]}^2.
\end{equation*}
In particular, if $X$ has unique Fréchet mean $\mu$ and $X$ is supported in the geodesic ball $(\mathcal{B}_\mu,\log_{\mu})$, then 
\begin{equation*}
\text{Rcov}_{\mu}(X,X)=\mathcal{F}_2(\mu).
\end{equation*}
\end{proposition}
\begin{proof}
See appendix \ref{sec:appendix}.
\end{proof}
Observe that this formula is analogous to the one for the variance in Euclidean space,
\begin{equation*}
Var(X)=E[X^2]-E^2[X].
\end{equation*}
If $X$ and $Y$ are both $M$-valued random variables with the same Fréchet mean $\mu$, then one can easily see by following the same argument in the proof of Proposition \ref{proposition Rcov(X,X) and Frechet function simple relation} that
\begin{equation*}
\text{Rcov}_\mu(X,Y)=E[\log_{\mu}X^T\log_{\mu}Y].
\end{equation*}
Because of this observation, when $X$ and $Y$ share the same Fréchet mean $\mu$, a natural choice is to evaluate $\text{Rcov}_{p}(X, Y)$ at $p = \mu$.
On the other hand, if $X$ and $Y$ have different Fréchet means, the choice of $p$ becomes arbitrary.
However, we can obtain an interpretable value for $\text{Rcov}_{p}(X, Y)$ by selecting $p$ as the midpoint between the two Fréchet means, provided there is a unique length-minimising geodesic between them, as we will discuss in subsection \ref{sec:estimators}.

\subsection{Riemannian correlation}\label{Riemannian correlation section}
Based on the covariance we have defined, the next step is to derive a correlation measure from it. 
This Riemannian correlation provides an analogue of the classical Pearson correlation for Riemannian-valued random variables.
\begin{definition}[Riemannian correlation]
Suppose that $X$ and $Y$ are $M$-valued and assume $p\in\mathcal{D}(X,Y)$.
We define their Riemannian cross-correlation matrix at $p$ to be
\begin{equation*}
\mathcal{R}_p(X,Y)=\frac{\Sigma_{p}(X,Y)}{\sqrt{tr(\Sigma_{p}(X,X))} \sqrt{tr(\Sigma_{p}(Y,Y))}}.
\end{equation*}
The Riemannian Pearson-correlation, or simply the Riemannian correlation, is defined as
\begin{equation*}
\text{Rcorr}_p(X,Y)=tr(R_p(X,Y)).
\end{equation*}
\end{definition}
The immediate natural question is whether Rcorr$_p(X,Y)$ is bounded between $-1$ and $1$ as in the case of classical Pearson-correlation. This is indeed the case, as shown in the following proposition.

\begin{proposition}\label{proposition: Rcorr is between -1 and 1}
Let $M$ be a compact and connected Riemannian manifold. Let $X$ and $Y$ be two $M$-valued random variables and $p\in\mathcal{D}(X,Y)$.
Then $\text{Rcorr}_{p}(X,Y)\in[-1,1]$.
\end{proposition}
\begin{proof}
See appendix \ref{sec:appendix}.
\end{proof}
As it was in the case of the Riemannian covariance, it is more convenient to deal with Rcorr than $\mathcal{R}$, due to its coordinate-independence. Like the Pearson correlation in statistics, $\text{Rcorr}_p(X,Y)$ measures the collinearity of the projection of $X$ and $Y$ in the tangent space at $P$.

The statistical interpretation of Rcorr is similar to that in Euclidean spaces, with lines being replaced by geodesics. If $\text{Rcorr}_p(X,Y)$ is positive, it means that $X$ and $Y$ tend to move together in the same direction along some radial geodesic at $p$. On the other hand, if $\text{Rcorr}_p(X,Y)$ is negative, this indicates that $X$ and $Y$ tend to move in opposite directions along some radial geodesic at $p$.
The absolute value of Rcorr captures the strength of the alignment along geodesics.

\subsection{Estimating Riemannian covariance and Riemannian correlation from a paired sample}
\label{sec:estimators}
Now that we have defined the Riemannian covariance and correlation, we turn to their estimators.
Let us assume that $\set{(X_k,Y_k)}{k=1,\ldots,N}$ is a random sample from a pair of $M$-valued random variables $(X,Y)$.
A natural estimator for $\Sigma_{p}(X,Y)$ is
\begin{equation*}
\widehat{\Sigma}_{p}(X,Y) = \frac{1}{N}\sum_{k=1}^N \left(\log_{p}X_k\right)\left(\log_{p}Y_k\right)^T - \left(\frac{1}{N}\sum_{k=1}^N\log_{p}X_k\right)\left(\frac{1}{N}\sum_{k=1}^N\log_{p}Y_k\right)^T.
\end{equation*}
So an estimator for $\text{Rcov}_p(X,Y)$ is the sample Riemannian covariance
\begin{equation*}
\widehat{\text{Rcov}}_p(X,Y)=\operatorname{tr}(\widehat{\Sigma}_{p}(X,Y)).
\end{equation*}
Moreover, an estimator for the Riemannian cross-correlation matrix can be obtained as
\begin{equation*}
\widehat{\mathcal{R}}_{p}(X,Y)=\frac{\widehat{\Sigma}_{p}(X,Y)}{\sqrt{\operatorname{tr}(\widehat{\Sigma}_{p}(X,X))} \sqrt{\operatorname{tr}(\widehat{\Sigma}_{p}(Y,Y))}},
\end{equation*}
and the corresponding estimator for $\text{Rcorr}$ is the sample Riemannian correlation
\begin{equation*}
\widehat{\text{Rcorr}}_{p}(X,Y)=\operatorname{tr}(\widehat{\mathcal{R}}_p(X,Y))=\frac{\widehat{\text{Rcov}}_p(X,Y)}{\sqrt{\widehat{\text{Rcov}}_p(X,X)} \sqrt{\widehat{\text{Rcov}}_p(Y,Y)}}.
\end{equation*}

As discussed in Section \ref{Covariance in convex geodesic balls section}, if the random objects $X$ and $Y$ share a common Fréchet mean $\mu$, it is convenient to evaluate the Riemannian covariance between them at their common mean $\mu$, as remarked in Proposition \ref{proposition Rcov(X,X) and Frechet function simple relation}.
This provides the usual interpretation of covariance and correlation in Euclidean spaces as measures of co-variability around the mean.
In the case of distinct Fréchet means, $\mu$ and $\nu$, for $X$ and $Y$, respectively, a natural point of evaluation for the Riemannian covariance is at the midpoint of the geodesic between the two means.
This does not have an immediate statistical interpretation, as in the case of the common mean, but it provides us with a local description where we can expect both variables to be well represented in a common tangent space.
This motivates the evaluation of measures at a point $p$ based on the Fréchet means. However, in practice, the true Fréchet mean(s) are not known. Therefore, we need to estimate the chosen point $p$ using an estimator $p_N$ based on the sample $(X_1, Y_1), \ldots, (X_N, Y_N)$. For instance, if $X$ and $Y$ have a unique common Fr\'echet mean $\mu$, we are interested in the Riemannian covariance $\text{Rcov}_{\mu}(X,Y)$. 
In this case, $p_N$ would be chosen as the sample Fr\'echet mean ${\widehat{\mu}_{N}}$ of the $2N$ observations $X_1, \ldots, X_N, Y_1, \ldots, Y_N$, and the estimator for $\text{Rcov}_{\mu}(X,Y)$ is given by
\begin{equation*}
\widehat{\text{Rcov}_{\mu}}(X,Y) = \widehat{\text{Rcov}}_{\widehat{\mu}_{N}}(X,Y).
\end{equation*}
More generally, if $p$ is a point on the geodesic between the Fr\'echet mean $\mu$ of $X$ and the Fr\'echet mean $\nu$ of $Y$, we can choose $p_N$ as the corresponding point on the geodesic between $\widehat{\mu}_N$ and $\\widehat{nu}_N$, where $\mu_N$ is the sample Fr\'echet mean of $X_1, \ldots, X_N$, and $\widehat{\nu}_N$ is the sample Fr\'echet mean of $Y_1, \ldots, Y_N$. The estimator for $\Sigma_{p}(X,Y)$ is then given by
\begin{equation*}
\widehat{\text{Rcov}_{p}}(X,Y) = \widehat{\text{Rcov}}_{\widehat{p}_{N}}(X,Y).
\end{equation*}
In particular, we take $\widehat{p}_N$ to be estimate for the midpoint between the means. In the next section, we will consider the convergence aspects of these estimators.

\section{Consistency results}\label{Different means: Common tangent space section}
We now discuss the consistency of $\widehat{\text{Rcov}}_{p_N}(X,Y)$ and $\widehat{\text{Rcorr}}_{p_N}(X,Y)$ as estimators for $ \text{Rcov}_{p}(X,Y)$ and $ \text{Rcorr}_{p}(X,Y)$, respectively. 
We will prove that, if $p_N$ converges to $p$ almost surely, then $\widehat{\Sigma}_{p_N}(X,Y)$ converges almost surely to $\Sigma_p(X,Y)$.
Before we proceed to prove this claim, we need the following lemmas.
\begin{lemma}\label{lemma: boundedness of Frechet function}
Let $M$ be a compact and connected Riemannian manifold and $X\colon\Omega\to M$ be random object with probability measure $Q_X$. Then for all $r>0$, $\mathcal{F}_{r,X}(p)$ attains both minimum and maximum. In particular, the Frechet function is bounded on $M$.
\end{lemma}
\begin{proof}
See appendix \ref{sec:appendix}.
\end{proof}
It is worth mentioning that the proof, as presented in appendix \ref{sec:appendix}, does not use the full Riemannian geometry of $M$.
Specifically, the same argument can be used to prove the claim in the setting of compact and connected metric spaces equipped with a finite measure.

\begin{lemma}\label{lemma on almost sure convergence of log map}
Let $M$ be a compact and connected Riemannian manifold and assume that $X$ is an $M$-valued random variable.
Suppose that $X$ is compactly supported in a geodesic ball $\mathcal{B}_{p}(r)$, where $r<\text{inj}(p)$ for some $p\in M$.
If $p_n$ is a sequence that converges almost surely to $p$, then
\begin{equation*}
\log_{p_n}X\overset{a.s.}{\longrightarrow}\log_{p}X,
\end{equation*}
as $n\to\infty$.
\end{lemma}
\begin{proof}
See appendix \ref{sec:appendix}.
\end{proof}
The assumption that $r<$inj$(p)$ is sufficient to avoid concentration near the boundary of the maximal geodesic ball at $p$.
Now we are in position to state the main theorem and prove it.
\begin{theorem}\label{main theorem: strong consistency for Riemannian covariance}
Let $M$ be a compact and connected Riemannian $n$-manifold. Suppose that $X$ and $Y$ are $M$-valued random variables. Assume that $X$ and $Y$ are compactly supported in the geodesic ball $\mathcal{B}_{p}(r)$, where $r < \text{inj}(p)$ for some $p \in M$. Let $\{(X_k,Y_k)\}_{k=1,\ldots,N}$ be an independently and identically distributed random sample, and for each $N$, let $p_N$ be an $M$-valued function of the sample that converges almost surely to $p$ as $N \to \infty$. Then,
\begin{equation*}
\widehat{\Sigma}_{p_N}(X,Y) \overset{a.s.}{\longrightarrow} \Sigma_p(X,Y),
\end{equation*}
as $N\to\infty$. In particular, $\widehat{\text{Rcov}}_{p_N}(X,Y) \overset{a.s.}{\longrightarrow} \text{Rcov}_p(X,Y)$ as $N \to \infty$.
\end{theorem}

\begin{proof}
For each $p_N$, the estimator $\widehat{\Sigma}_{p_{N}}(X,Y)$ is given by
\begin{equation*}
\widehat{\Sigma}_{p_{N}}(X,Y) = \frac{1}{N} \sum_{k=1}^N \left( \log_{p_N} X_k \right) \left( \log_{p_N} Y_k \right)^T - \left( \frac{1}{N} \sum_{k=1}^N \log_{p_N} X_k \right) \left( \frac{1}{N} \sum_{k=1}^N \log_{p_N} Y_k \right)^T.
\end{equation*}

The $ij$th element of $\widehat{\Sigma}_{p_{N}}(X,Y)$ is
\begin{equation*}
\widehat{\Sigma}_{p_{N}}(X,Y)_{ij} = \frac{1}{N} \sum_{k=1}^N (\log_{p_N} X^i_k) (\log_{p_N} Y^j_k) - \left( \frac{1}{N} \sum_{k=1}^N \log_{p_N} X^i_k \right) \left( \frac{1}{N} \sum_{k=1}^N \log_{p_N} Y^j_k \right).
\end{equation*}

We aim to prove $\widehat{\Sigma}_{p_N}(X,Y)_{ij} \overset{a.s.}{\longrightarrow} \Sigma_{p}(X,Y)_{ij}$, for all $i,j=1,\ldots,n$.

As $p_{N}$ converges almost surely to $p$, from Lemma \ref{lemma on almost sure convergence of log map} and the continuous mapping theorem, for all $k\in\mathbb{N}$, $\log_{p_N}X_k \overset{a.s.}{\to} \log_p X_k$, $\log_{p_N} Y_k \overset{a.s.}{\to} \log_p Y_k$, and 
\begin{equation*}
(\log_{p_N} X_k)^T (\log_{p_N} Y_k) \overset{a.s.}{\longrightarrow} \left( \log_{p} X_k \right)^T \left( \log_{p} Y_k \right)
\end{equation*}
as $N \to \infty$. This implies the almost sure convergence component-wise
\begin{equation*}
\log_{p_N} X^i_k \log_{p_N} Y^j_k \overset{a.s.}{\longrightarrow} \log_{p} X_k^i \log_{p} Y_k^j,
\end{equation*}
as $N \to \infty$.
Note that, using Cauchy-Schwarz inequality and Lemma \ref{lemma: boundedness of Frechet function}
\begin{equation*}
E \left[ \left| \log_{p} X_{k}^i \log_{p} Y^{j}_{k} \right| \right] \leq E \left[ \left| \log_{p} X_{k}^i \right|^2 \right]^{1/2} E \left[ \left| \log_{p} Y_{k}^j \right|^2 \right]^{1/2} = \mathcal{F}_{2,X}(p)^{1/2} \mathcal{F}_{2,Y}(p)^{1/2} < \infty.
\end{equation*}
Thus, we can invoke Etemadi's strong law of large numbers,
\begin{equation} \label{strong-law-of-numbers-for-c_k}
\frac{1}{N} \sum_{k=1}^{N} \left[ \log_{p} X_{k}^i \log_{p} Y_{k}^j \right] \overset{a.s.}{\longrightarrow} E [\log_{p} X^i_1 \log_{p} Y^j_1] = \Sigma_{p}(X,Y)_{ij}.
\end{equation}
    
On the other hand, using Lemma \ref{lemma: boundedness of Frechet function}
\begin{equation*}
E [\left| \log_p X_k^i \right|]\leq E [\| \log_p X_k \|^{1/2}] \leq \mathcal{F}_{\frac{1}{2},X}(p) < \infty.
\end{equation*}
So, by the Strong Law of Large Numbers,
\begin{equation*}
\frac{1}{N} \sum_{k=1}^N \log_{p} X_k^i \overset{a.s.}{\longrightarrow} E [\log_{p} X^i_1],
\end{equation*}
as $N \to \infty$, for each $i=1,\ldots,n$. A similar observation holds for $\frac{1}{N} \sum_{k=1}^N \log_{p} Y_k^{j}$, ${j=1,\ldots,n}$. So if we can prove
\begin{equation} \label{conv-of-log_p_N-X-log_p_N-Y}
\frac{1}{N} \sum_{k=1}^N \log_{p_N} X_k^i \log_{p_N} Y_k^j \overset{a.s.}{\longrightarrow} \frac{1}{N} \sum_{k=1}^N \log_{p} X_k^i \log_{p} Y_k^{j},
\end{equation}
\begin{equation} \label{conv-of-log_p_N-X-log_p_N-X}
\frac{1}{N} \sum_{k=1}^N \log_{p_N} X^i_k \overset{a.s.}{\longrightarrow} \frac{1}{N} \sum_{k=1}^N \log_{p} X^i_k,
\end{equation}
\begin{equation} \label{conv-of-log_p_N-Y-log_p_N-Y}
\frac{1}{N} \sum_{k=1}^N \log_{p_N} Y^j_k \overset{a.s.}{\longrightarrow} \frac{1}{N} \sum_{k=1}^N \log_{p} Y^j_k,
\end{equation}
the assertion of the theorem follows, using the strong law of large numbers. We prove (\ref{conv-of-log_p_N-X-log_p_N-Y}) as (\ref{conv-of-log_p_N-X-log_p_N-X}) and (\ref{conv-of-log_p_N-Y-log_p_N-Y}) follow similar argument. For simplicity, put
\begin{equation*}
a_{N,k}^{ij} = (\log_{p_N} X^i_k) (\log_{p_N} Y^j_k),
\end{equation*}
and
\begin{equation*}
c_{k}^{ij} = \left( \log_{p} X^i_k \right) \left( \log_{p} Y^j_k \right).
\end{equation*}
If we can show that as $N \to \infty$
\begin{equation} \label{strong-consistency-transitive-convergence-statement}
\frac{1}{N} \sum_{k=1}^{N} a^{ij}_{N,k} \overset{a.s.}{\longrightarrow} \frac{1}{N} \sum_{k=1}^{N} c^{ij}_{k},
\end{equation}
the claim of the theorem follows, since we have established that
\begin{equation*}
\frac{1}{N} \sum_{k=1}^{N} c^{ij}_{k} = \frac{1}{N} \sum_{k=1}^{N} \left[ \log_{p} X_{k}^i \log_{p} Y_{k}^j \right] \overset{a.s.}{\longrightarrow} E [\log_{p} X^i_1 \log_{p} Y^j_1] = \Sigma_{p}(X,Y)_{ij},
\end{equation*}
as in (\ref{strong-law-of-numbers-for-c_k}). We need to show is that for all $\varepsilon > 0$, there exists $A \in \Sigma$ with $P(A) = 0$ such that for any $\omega \in \Omega \setminus A$ there exists $N^* \in \mathbb{N}$ such that for all $N > N^*$
\begin{equation*}
P \left( \left| \frac{1}{N} \sum_{k=1}^N a^{ij}_{k,N}(\omega) - \frac{1}{N} \sum_{k=1}^N c^{ij}_{k}(\omega) \right| < \varepsilon \right) = P \left( \left| \frac{1}{N} \sum_{k=1}^N \left( a^{ij}_{k,N}(\omega) - c^{ij}_{k}(\omega) \right) \right| < \varepsilon \right) = 1.
\end{equation*}
By set inclusions, we have the estimate
\begin{equation} \label{strong-consistency-common-mean-estimate-for-probability-1}
P \left( \left| \frac{1}{N} \sum_{k=1}^N a^{ij}_{k,N}(\omega) - c^{ij}_k(\omega) \right| < \varepsilon \right) \geq P \left( \frac{1}{N} \sum_{k=1}^N \left| a^{ij}_{k,N}(\omega) - c^{ij}_k(\omega) \right| < \varepsilon \right),
\end{equation}
for any $\varepsilon > 0.$ By Lemma \ref{lemma on almost sure convergence of log map}, we know that there exists a $P$-null set $A$ such that for all $\omega \in \Omega \setminus A$ and for all $\varepsilon^* > 0$, there exists $N^*(\omega) \in \mathbb{N}$ such that
\begin{equation*}
\left| a^{ij}_{N,k}(\omega) - c^{ij}_{k}(\omega) \right| < \varepsilon^*,
\end{equation*}
for all $N > N^*(\omega)$ and for all $k \leq N$. Now this implies
\begin{equation*}
\frac{1}{N} \sum_{k=1}^{N} \left| a^{ij}_{N,k}(\omega) - c^{ij}_{k}(\omega) \right| < \varepsilon^*,
\end{equation*}
for $\omega \in \Omega \setminus A$ and $N > N^*(\omega)$. By setting $\varepsilon^* < \varepsilon$, we have
\begin{equation*}
P \left( \frac{1}{N} \sum_{k=1}^N \left| a^{ij}_{N,k}(\omega) - c^{ij}_k(\omega) \right| < \varepsilon \right) \geq P \left( \frac{1}{N} \sum_{k=1}^N \left| a^{ij}_{N,k}(\omega) - c^{ij}_k(\omega) \right| < \varepsilon^* \right) = 1,
\end{equation*}
for all $N > N^{*}(\omega)$. This gives
\begin{equation*}
P \left( \frac{1}{N} \sum_{k=1}^N \left| a^{ij}_{N,k}(\omega) - c^{ij}_k(\omega) \right| < \varepsilon \right) = 1,
\end{equation*}
for all $N > N^*$ and $\omega \in \Omega \setminus A$. Using this in estimate (\ref{strong-consistency-common-mean-estimate-for-probability-1}), we get
\begin{equation*}
P \left( \left| \frac{1}{N} \sum_{k=1}^N \left( a^{ij}_{N,k}(\omega) - c^{ij}_k(\omega) \right) \right| < \varepsilon \right) = 1,
\end{equation*}
for all $N > N^{*}(\omega)$, where $\omega \in \Omega \setminus A$. So,
\begin{equation*}
\frac{1}{N} \sum_{k=1}^N a^{ij}_{N,k} \overset{a.s.}{\longrightarrow} \frac{1}{N} \sum_{k=1}^{N} c^{ij}_{k}.
\end{equation*}
Therefore, by (\ref{strong-consistency-transitive-convergence-statement}),
\begin{equation*}
\widehat{\Sigma}_{p_{N}}(X,Y)_{ij} = \frac{1}{N} \sum_{k=1}^N a^{ij}_{N,k} \overset{a.s.}{\longrightarrow} E [c^{ij}_1].
\end{equation*}
The same argument holds for $a_{N,k}^i = \log_{p_N} X^i_k$ and $c_k^{i} = \log_{p} X^i_k$ and also for $Y_k^j$, $i,j=1,\ldots,n$. So we have $\widehat{\Sigma}_{p_{N}}(X,Y) \overset{a.s.}{\longrightarrow} \Sigma_{p}(X,Y)$. Finally, the convergence of $\widehat{Rcov}_{p_N}$ follows from the continuity of the trace as a function.
\end{proof}

\begin{corollary}\label{strong consistency for correlations}
Under the same assumptions of Theorem \ref{main theorem: strong consistency for Riemannian covariance},
The estimator $\widehat{\mathcal{R}}_{p_N}(X,Y)\overset{a.s.}{\longrightarrow}\mathcal{R}_p(X,Y)$ and $\widehat{\text{Rcorr}}_{p_N}(X,Y)\overset{a.s.}{\longrightarrow}\text{Rcorr}_{p}(X,Y)$, as $N\to\infty$.
\end{corollary}
\begin{proof}
Both $\widehat{\mathcal{R}}_{p_N}$ and $\widehat{\text{Rcorr}}_{p_N}$ are the compositions of continuous functions of $\widehat{\Sigma}_{p_N}$ and $\widehat{\text{Rcov}}_{p_N}$.
By Theorem \ref{main theorem: strong consistency for Riemannian covariance} and the continous mapping theorem the claim follows.
\end{proof}
Let's now consider the two scenarios for the Fréchet means and the choices for $p$ as we discussed in Section \ref{sec:estimators}. 
The first scenario is when the random objects share a common and unique Fréchet mean.
\begin{proposition}\label{Rcorr at Frechet sample mean convergence}
Let $M$ be a compact and connected Riemannian $n$-manifold. 
Suppose $X$ and $Y$ are $M$-valued random variables with a common unique Fréchet mean $\mu$ such that $X$ and $Y$ are compactly supported in the geodesic ball $\mathcal{B}_{\mu}(r)$, where $r<\text{inj}(\mu)$. If $\set{(X_k,Y_k)}{k=1,\ldots,N}$ is a random sample and $\widehat{\mu}_N$ is the Fréchet sample mean, then
\begin{equation*}
\widehat{\Sigma}_{\widehat{\mu}_N}(X,Y)\overset{a.s.}{\longrightarrow}\Sigma_{\mu}(X,Y).
\end{equation*}
Similarly, $\widehat{\text{Rcov}}_{\widehat{\mu}_N}(X,Y)\overset{a.s.}{\rightarrow}\text{Rcov}_{\mu}(X,Y)$, $\widehat{\mathcal{R}}_{\widehat{\mu}_N}(X,Y)\overset{a.s.}{\rightarrow}\mathcal{R}_{\mu}(X,Y)$, and $\widehat{\text{Rcorr}}_{\widehat{\mu}_N}(X,Y)\overset{a.s.}{\rightarrow}\text{Rcorr}_{\mu}(X,Y)$.

\end{proposition}

\begin{proof}
Using the strong consistency result for $\widehat{\mu}_N$, Theorem 4.2.2 in \cite{patrangenaru2016nonparametric}, we know that $\widehat{\mu}_N\overset{a.s.}{\longrightarrow}\mu$. The claims follow by taking $p_N=\widehat{\mu}_N$ in Theorem \ref{main theorem: strong consistency for Riemannian covariance} and Corollary \ref{strong consistency for correlations}.
\end{proof}
On the other hand, if the Fréchet means of $X$ and $Y$ are $\mu$ and $\nu$, respectively, we aim to estimate the Riemannian covariance and correlation at $m$, where $m$ is the midpoint between $\mu$ and $\nu$.
However, in order for the midpoint to be uniquely defined, we impose some condition on the Fréchet means.
Namely, we assume that $\mu$ and $\nu$ lie within an open convex set.
\begin{proposition}\label{Rcorr at midpoint of Frechet sample menas convergence}
Let $M$ be a compact and connected Riemannian $n$-manifold. Suppose $X,Y\colon\Omega\to M$ are $M$-valued random variables such that $\mu$ is the unique Fréchet mean of $X$ and $\nu$ is the unique Fréchet mean of $Y$.
Suppose that $\mu$ and $\nu$ are contained in an open convex set $\mathcal{C}$.
Let $\gamma\colon[0,1]\to M$ is the unique length-minimising geodesic connecting $\mu$ and $\nu$ with midpoint $m=\gamma(1/2)$ such that both $X$ and $Y$ are compactly supported in $\mathcal{B}_m(r)$ where $r<\text{inj}(m)$.
Assume that $\set{(X_k,Y_k)}{k=1,\ldots,N}$ is a random sample, $\widehat{\mu}_N$ is the Fréchet sample mean of $X_1,\ldots,X_N$, and $\widehat{\nu}_N$ is the Fréchet sample mean of $Y_1,\ldots,Y_N$.
If $\widehat{m}_N$ is the midpoint of $\widehat{\mu}_N$ and $\widehat{\nu}_N$, then
\begin{equation*}
\widehat{\Sigma}_{\widehat{m}_N}(X,Y)\overset{a.s.}{\longrightarrow}\Sigma_{m}(X,Y).
\end{equation*}
Similarly, $\widehat{\text{Rcov}}_{\widehat{m}_N}(X,Y)\overset{a.s.}{\rightarrow}\text{Rcov}_{m}(X,Y)$, $\widehat{\mathcal{R}}_{\widehat{m}_N}(X,Y)\overset{a.s.}{\rightarrow}\mathcal{R}_{m}(X,Y)$, and $\widehat{\text{Rcorr}}_{\widehat{m}_N}(X,Y)\overset{a.s.}{\rightarrow}\text{Rcorr}_{m}(X,Y)$.
\end{proposition}

\begin{proof}
Since the Fréchet means $\mu$ and $\nu$ are contained in the open convex set $\mathcal{C}$,  there exist neighborhoods $U$ of $\mu$ and $V$ of $\nu$ such that for all $p \in U$ and $q \in V$, there exists a unique length-minimising geodesic between them. Consequently, there exists a unique midpoint between $p$ and $q$. 
Since $\widehat{\mu}_N$ and $\widehat{\nu}_N$ converge almost surely to $\mu$ and $\nu$, respectively, there exists $N_0$ such that for all $N > N_0$ we have $P(\widehat{\mu}_N \in U) = P(\widehat{\nu}_N \in V) = 1$.

Now consider the sequence of length-minimising geodesics $\gamma_N(t)$ from $\widehat{\mu}_N$ to $\widehat{\nu}_N$. Assume that the midpoint, $\widehat{m}_N$, of $\gamma_N(t)$ is attained when $t = 1/2$. Then the midpoint of $\gamma_N(t)$ is given by
\begin{equation*}
\gamma_N(1/2) = \widehat{m}_N = \exp_{\widehat{\mu}_N}\left(\frac{1}{2}\log_{\widehat{\mu}_N}\widehat{\nu}_N\right).
\end{equation*}
Note that the function
\begin{equation*}
h(p, q) = \exp_{p}\left(\frac{1}{2}\log_{p}q\right),
\end{equation*}
is continuous on $\mathcal{C} \times \mathcal{C}$. Since $\widehat{m}_N = h(\widehat{\mu}_N, \widehat{\nu}_N)$ and $m = h(\mu, \nu)$, the continuity of $h$ implies $\widehat{m}_N \overset{a.s.}{\longrightarrow} m$ as $N \to \infty$.

Because the supports of $X$ and $Y$ are contained in $\mathcal{B}_m(r)$, $r < \text{inj}(m)$, the requirements of Theorem \ref{main theorem: strong consistency for Riemannian covariance} hold for $p_N = \widehat{m}_N$. Therefore, the almost sure convergence of $\widehat{\Sigma}_{\widehat{m}_N}(X, Y)$ to $\Sigma_{m}(X, Y)$ holds. The convergence of $\widehat{\text{Rcov}}_{\widehat{m}_N}(X, Y)$, $\widehat{\mathcal{R}}_{\widehat{m}_N}(X, Y)$, and $\widehat{\text{Rcorr}}_{\widehat{m}_N}(X, Y)$ follows from Theorem \ref{main theorem: strong consistency for Riemannian covariance} and Corollary \ref{strong consistency for correlations}.
\end{proof}

As a variation of the midpoint approach for the case of different Fréchet means, we could choose to compute the Riemannian covariance and correlation at a weighted average of the two means. Namely, rather than considering the midpoint $\gamma(1/2)$ between $\mu$ and $\nu$, we could consider $\gamma\left(\frac{w_1}{w_1+w_2}\right)$. This can be useful for example when the weights $w_1$ and $w_2$ are determined by the Fréchet total variances of $X$ and $Y$, $\mathcal{F}_X(\mu)$ and $\mathcal{F}_Y(\nu)$. Under such a choice, a  result similar to Proposition \ref{Rcorr at midpoint of Frechet sample menas convergence} clearly holds.

\section{Empirical demonstrations}
In this section, we present some simulation studies and real-world data examples to illustrate the effectiveness of the Riemannian correlation as a measure of stochastic dependence.

In the simulations, we describe examples on two manifolds, the first being the unit 2-sphere, $\mathbb{S}^2$, and the other being $SO(3)$, the special orthogonal group of $3 \times 3$ matrices.
The aim of the simulation studies is to explore the finite sample behavior of the sample Riemannian correlation when the strength of the stochastic dependence changes, and in the case where the two samples are independent by construction.
In the first scenario, we generate pairs of random variables with inherent strong dependence and then introduce random perturbations, with increasing perturbation weakening the dependence between the two samples. 
In the second case, we construct a scenario where the two samples are fully independent.
Across these simulated examples, we compare the behaviour of the sample Riemannian correlation (Rcorr) with distance correlation in metric spaces (dcorr, \cite{distance-covariance-in-metric-spaces-Lyons}).
After the simulations, we proceed with a real-world example, where we apply Rcorr to vectorcardiogram datasets.

\subsection{Simulation study for $\mathbb{S}^2$-data}
\subsubsection{Generative Model}

Our manifold of interest here is $\mathbb{S}^2$ with the round metric.
We start by generating random dataset on $\mathbb{S}^2$ using the von Mises distribution, also known as von Mises-Fisher distribution.
An $\mathbb{S}^2$-valued random variable $X$ is to have von Mises-Fisher distribution with concentration $\mu\in\mathbb{S}^2$ and dispersion parameter $\kappa>0$, VMF($\mu,\kappa$), if its density is given by
\begin{equation*}
f(x)=\frac{\kappa}{4\pi \sinh\kappa}\exp\left(\kappa\mu^T x\right),
\end{equation*}
for all $x\in\mathbb{S}^2$.
We generate an initial dataset with a given mean and concentration. 
Then, we will construct a dependent dataset using a geometric transformations that include some noise.

Let us fix a choice of parameters $\mu_0\in\mathbb{S}^2$ and $\kappa_0>0$.
Using the sampling technique as given in \cite{jakob2012numerically}, we generate the initial sample $X=\{X_1,\ldots,X_N\}\sim\text{VMF}(\mu_0,\kappa_0)$.
In order the guarantee the convergence of Fréchet sample mean to a unique point, we chose the concentration parameter $\kappa_0$ such that the sample is contained within a convex ball centred at $\mu_0$.
To generate the second sample, whose dependence on the initial sample is controlled, we first choose a rotation  $R\in SO(3)$.
For computational convenience, we represent rotations by quaternions. A rotation in three-dimensional space is completely determined by the axis of rotation and the angle of rotation about that axis. Specifically, if $R$ is a rotation with unit axis of rotation $A=(\alpha_1,\alpha_2,\alpha_3)$ and $\theta$ is the angle of rotation, then $R$ can be expressed by the quaternion $q=a+bi+cj+dk$ as
\begin{equation*}
R = \begin{pmatrix}
a^2 + b^2 - c^2 - d^2 & 2(bc + ad) & 2(bd - ac) \\
2(bc - ad) & a^2 + c^2 - b^2 - d^2 & 2(cd + ab) \\
2(bd + ac) & 2(cd - ab) & a^2 + d^2 - b^2 - c^2
\end{pmatrix},
\end{equation*}
where $a=\cos(\theta/2)$, $b=\alpha_1\sin(\theta/2)$, $c=\alpha_2\sin(\theta/2)$, and $d=\alpha_3\sin(\theta/2)$. The theory behind rotations and quaternions can be found in \cite{hanson2006visualizing}.

After the applying the rotation $R$ to $X$, we obtain a sample which is deterministically dependent from the original sample. We then add a random perturbation to each data point in the following way.
For $i=1,\ldots,N$, let 
\begin{equation*}
b_{i}\sim N_3(0,\varepsilon^2I_3),
\end{equation*}
where $\varepsilon>0$ is the desired level of noise. The second sample $Y=\{Y_1,\ldots,Y_N\}$ is generated as
\begin{equation*}
Y_{i}=\frac{RX_i+b_i}{\norm{RX_i+b_i}}\in\mathbb{S}^2,
\end{equation*}
for $i=1,\ldots,N$. See Figure \ref{fig:dataset_on_unit_sphere_and_its_perturbation_version} for an example. Clearly, the two datasets are inherently dependent by construction, and this dependence weakens as the noise level in the perturbation increases.

Because of the spherical symmetry of the von Mises-Fisher distribution, we can control Rcorr by varying $\theta$.
For instance, with zero noise, we expect Rcorr to be 1 if $\theta = 0$, and we expect Rcorr to be $-1$ if $\theta = \pi$.
Additionally, in order for the two means to be within a convex set, as in the hypothesis of Proposition \ref{Rcorr at midpoint of Frechet sample menas convergence}, we chose the rotation $R$ so that $\{\mu, R\mu\}$ are not antipodal points.

We consider four different scenarios, which differ in the two samples having the same or different mean and in the type of dependence (or lack thereof) between them.
Based on Proposition \ref{Rcorr at Frechet sample mean convergence}, in the case of the same Fréchet mean, we compute Rcorr at the Fréchet sample mean. On the other hand, if the samples have different Fréchet means, we compute Rcorr at the midpoint between the Fréchet sample mean, as in Proposition \ref{Rcorr at midpoint of Frechet sample menas convergence}.
\begin{figure}[h]
    \centering
    \includegraphics[width=0.9\textwidth]{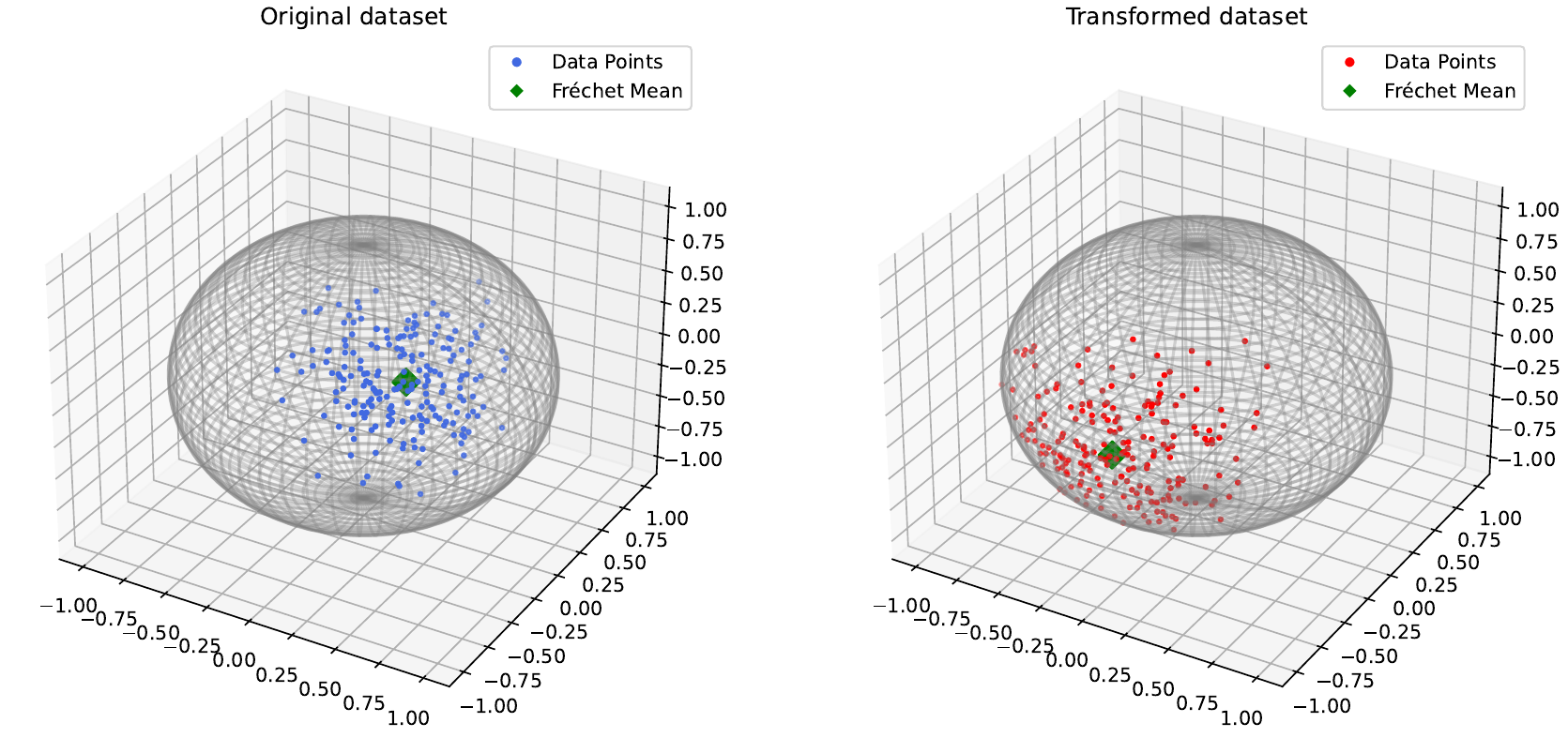}
    \caption{Two dependent datasets on $\mathbb{S}^2$, with the right being a rotated and perturbed version of the left.}
    \label{fig:dataset_on_unit_sphere_and_its_perturbation_version}
\end{figure}
\FloatBarrier

\subsubsection{Simulation results}
Figure \ref{fig:Rcorr_for_dependent_dataset_with_zero_rotation_on_unit_sphere} shows the estimated values of dcorr and Rcorr at various levels of noise, $\varepsilon$, when the initial parameters are chosen so that both datasets have the same Fréchet mean $(R=I_3)$. 
The plot shows that for low randomness both estimates correctly have similarly high value $(\approx 1)$. 
However, as the randomness grows, the sample Rcorr declines faster than the sample dcorr.
This captures how the dependence becomes less evident as the noise increases and the sample Rcorr gets closer to zero for higher level of noise.

Figure \ref{fig:Rcorr_for_dependent_dataset_with_nonzero_rotation_on_unit_sphere} 
we have the same scenario but with the different Fréchet mean, i.e. $R\neq I_3$. 
The plot indicates similar conclusions as in the case of the same Fréchet mean, showing that the proposed approach of using a midpoint is correctly capturing the dependence.

In Figure \ref{fig:negative_rcorr_on_unit_sphere} we consider a third scenario where the two samples have different Fréchet means and $\theta=\pi$, i.e., the two samples are now negatively correlated in the sense of Rcorr, as discussed in Section \ref{Riemannian correlation section}.
This is indeed captured by the sample Rcorr, which has negative values.
On the other hand, the sample dcorr is still positive, since it is a non-negative measure and it cannot distinguish between the two scenarios.

Finally, in Figure \ref{fig:Rcorr_for_independent_datasets_on_unit_sphere}, we have the comparison of these empirical measures for independent samples. Clearly, the sample Rcorr is capturing the independence better than the empirical dcorr. This is not totally surprising, since dcorr considers a broader set of potential dependencies, making it more sensitive to noise.

These simulations demonstrate that the proposed Riemannian correlation is indeed a good measure of dependence for sample on the unit sphere, with some advantages with respect to the distance correlation.

\begin{figure}[H]
    \centering
    \begin{subfigure}[t]{0.45\textwidth}
        \centering
        \includegraphics[width=\textwidth]{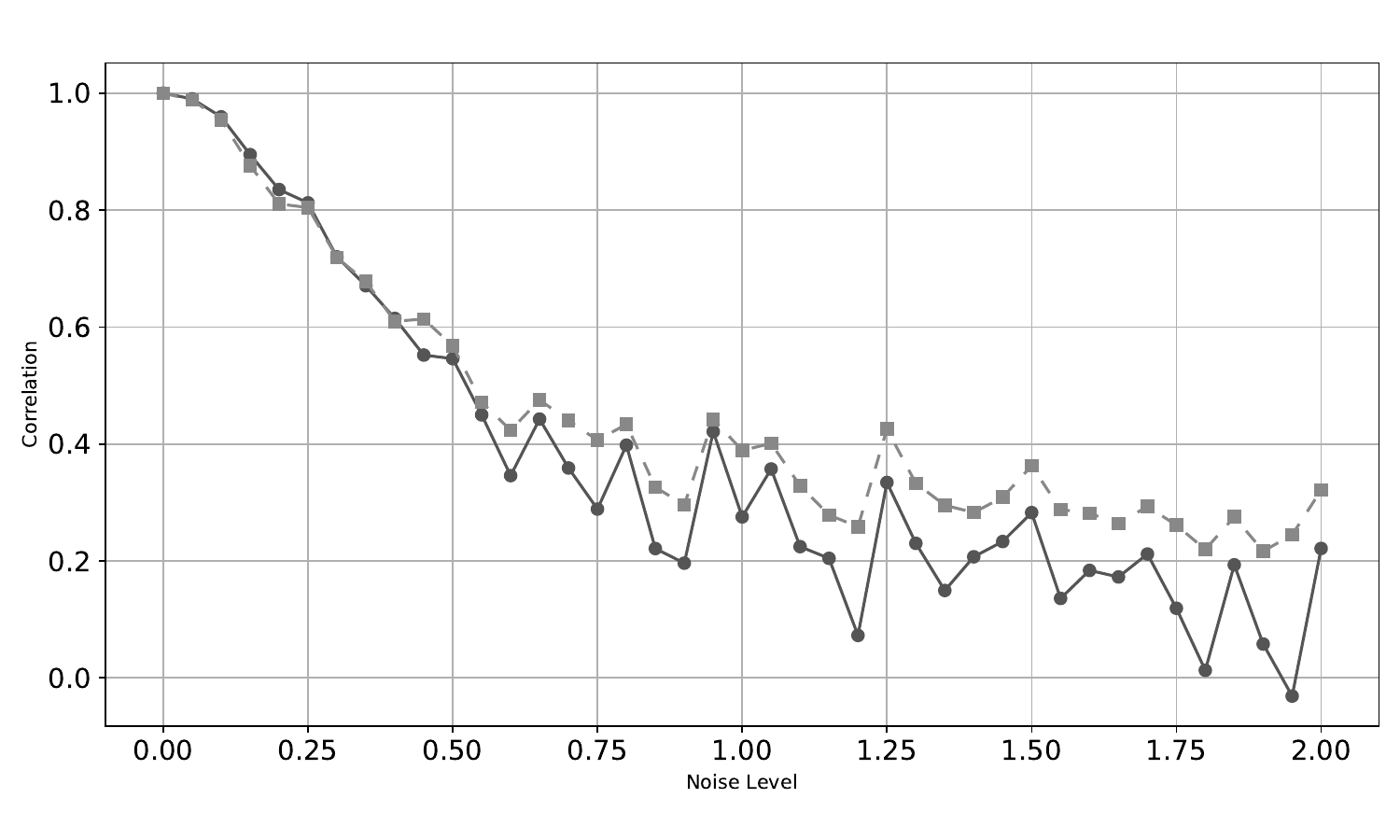}
        \caption{Comparison of Rcorr and dcorr against noise level for datasets on the unit sphere $\mathbb{S}^2$ with the same Fréchet mean. The parameters for the generative process are sample size $=100$, $\mu=(0,0,1)$, and $\kappa=9$.}
        \label{fig:Rcorr_for_dependent_dataset_with_zero_rotation_on_unit_sphere}
    \end{subfigure}
    \hfill
    \begin{subfigure}[t]{0.45\textwidth}
        \centering
        \includegraphics[width=\textwidth]{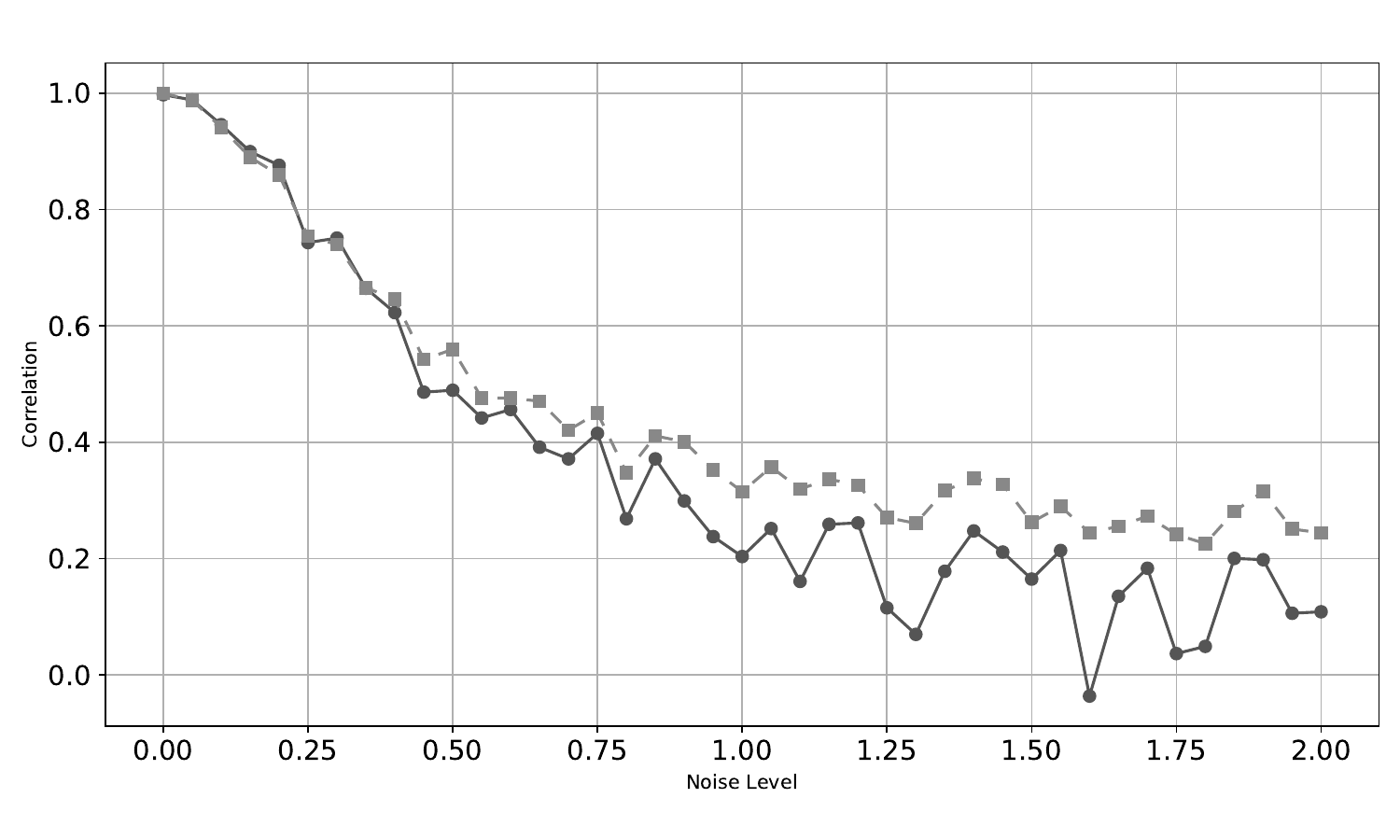}
        \caption{Comparison of Rcorr and dcorr against noise level for datasets on the unit sphere $\mathbb{S}^2$ with different Fréchet means. The parameters for the generative process are sample size $=100$, $\mu=(0,0,1)$, $\kappa=9$, axis of rotation $=(0,1,0)$, and angle of rotation $=\pi/6$.}
        \label{fig:Rcorr_for_dependent_dataset_with_nonzero_rotation_on_unit_sphere}
    \end{subfigure}
    \vskip\baselineskip
    \begin{subfigure}[t]{0.45\textwidth}
        \centering
        \includegraphics[width=\textwidth]{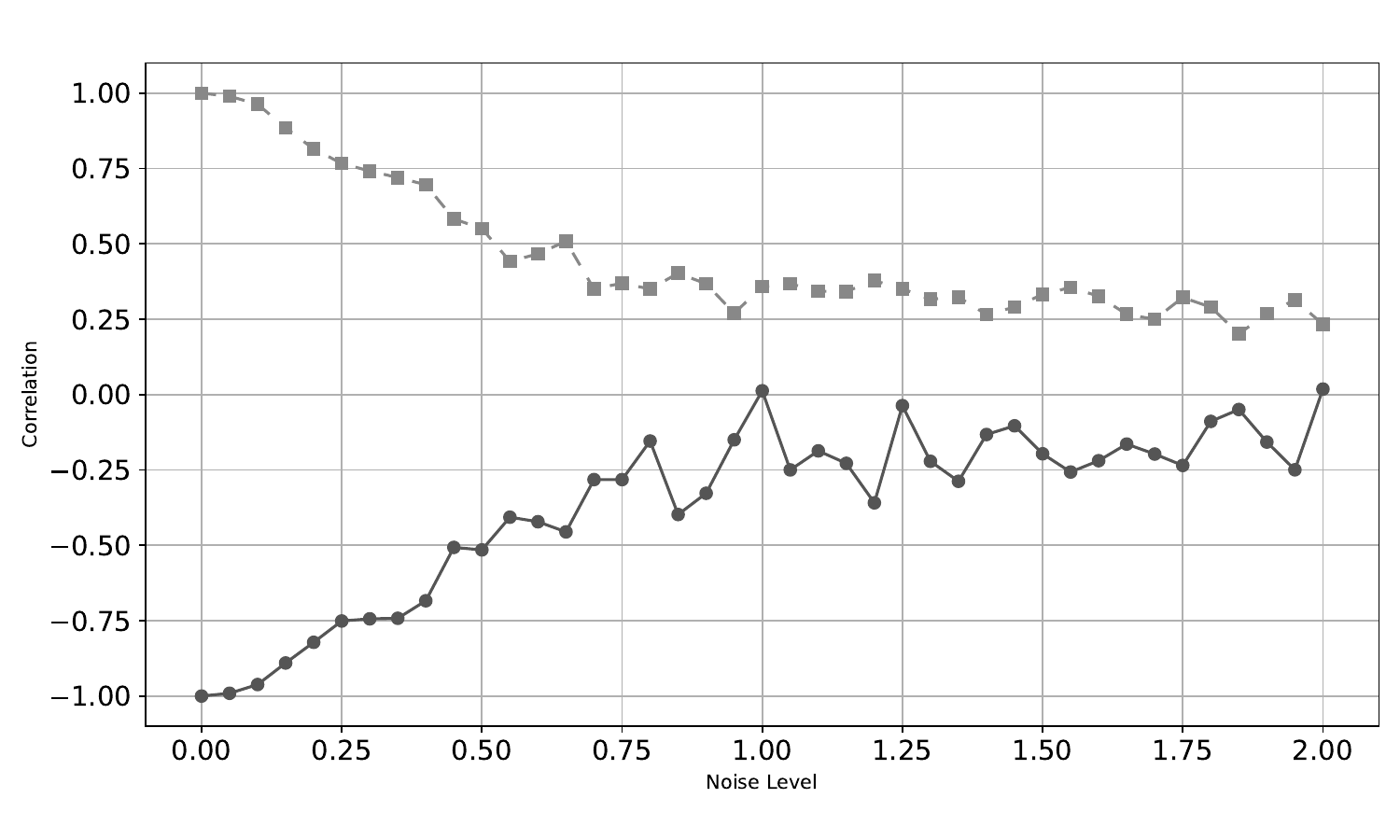}
        \caption{Comparison of Rcorr and dcorr against noise level for datasets on the unit sphere $\mathbb{S}^2$ with different Fréchet means. The parameters for the generative process are sample size $=100$, $\mu=(0,0,1)$, $\kappa=9$, axis of rotation $=(1,1,1)$, and angle of rotation $=\pi$.}
        \label{fig:negative_rcorr_on_unit_sphere}
    \end{subfigure}
    \hfill
    \begin{subfigure}[t]{0.45\textwidth}
        \centering
        \includegraphics[width=\textwidth]{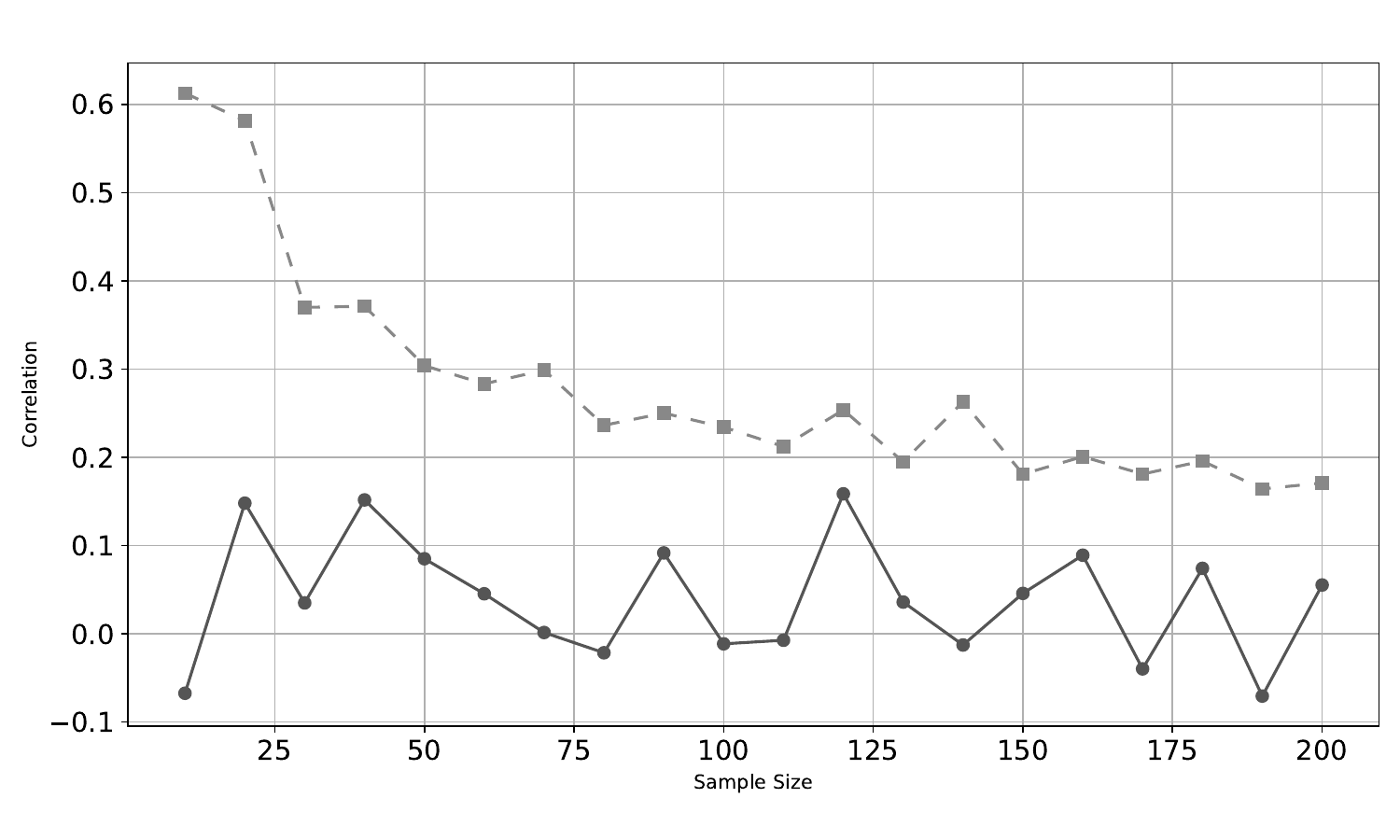}
        \caption{Comparison of Rcorr and dcorr for independent datasets on the unit sphere $\mathbb{S}^2$. The parameters for the generative process are ${\mu_1=(0,0,1)}, \mu_2=(0,1,1), \kappa_1=4,$ and $\kappa_2=5$.}
        \label{fig:Rcorr_for_independent_datasets_on_unit_sphere}
    \end{subfigure}
    \caption{Comparison of Rcorr (solid line with circular makers) and dcorr (dashed line with square markers) for datasets in $\mathbb{S}^2$ under different scenarios.}
    \label{fig:sphere_comparison_all}
\end{figure}
\subsection{Simulation study for $SO(3)$-data}
\subsubsection{Generative Model}
In this subsection, we consider the Lie group $SO$(3) as a Riemannian manifold  with the bi-invariant metric induced by the inner product $\inprod{X}{Y}_{I_3}=\frac{1}{2}tr(XY^T)$ on the tangent space at $I_3$, i.e., its Lie algebra $\mathfrak{so}(3)$.
Using the isomorphism
\begin{align}
\arrowmap{\Phi}{\mathbb{R}^3}{\mathfrak{so}(3)}{
\begin{pmatrix}
x \\
y \\
z
\end{pmatrix}
}{\begin{pmatrix}
0 & x & y \\
-x & 0 & z \\
-y & -z & 0
\end{pmatrix}},
\end{align}
we identify $\mathfrak{so}(3)$ with $\mathbb{R}^3$.
The norms of $A=(x,y,z)^T\in\mathbb{R}^3$ and $\Phi(A)\in\mathfrak{so}(3)$ are related by
\begin{equation}
\norm{A}=\frac{1}{\sqrt{2}}\norm{\Phi(A)}_{F},
\end{equation}
where $\norm{\cdot}$ is the Euclidean norm on $\mathbb{R}^3$ and $\norm{\cdot}_F$ is the Frobenius norm on $\mathfrak{so}(3)$.

As in the simulation for $\mathbb{S}^2$, we generate two samples which we have control of the their dependence. The generative process for the initial dataset $X = \{X_1, \ldots, X_N\}$ in $\mathfrak{so}(3)$ is as follows.
We generate a random set $A=\{A_1, \ldots, A_N\}$ in $\mathbb{R}^3$ such that such that $A_i\sim N_{3}(0,I_3)$ for each $i=1,\ldots,N$.
Using the isomorphism $\Phi$, we map $A$ to $\Phi(A)\subset\mathfrak{so}(3)$.
To ensure the convergence to a unique Fréchet sample mean, we impose a length bound on $\Phi(A_i)$ for each $i$. 
Since the convexity radius of $SO(3)$ is $\pi/2$, to ensure a unique sample mean, it suffices to choose a bound $\alpha$ such that $0<\alpha<\pi/2$.
If $\norm{\Phi(A_i)}_{F}>\alpha$, we rescale $\Phi(A_i)$ by $\frac{\alpha}{\norm{\Phi(A_i)}_{F}}$. 
Then, using the exponential map at the identity, we define our initial set as
\begin{equation*}
X_{i}=\exp(\Phi(A_i)),
\end{equation*}
for $i=1,\ldots,N$, where $\exp(\cdot)$ is the matrix exponential map.
Since the Riemannian metric on $SO(3)$ is bi-invariant, the matrix exponential map is precisely the Riemannian exponential map on the tangent space at the identity matrix for $SO(3)$. 
Note that by construction, the Fréchet sample mean of $X_1, \ldots, X_N$ converges to $I_{3}$ as $N\to\infty$, which is the true Frechet mean of $X$ as $SO(3)$-valued random variable.

To generate a sample $Y = \{Y_1, \ldots, Y_N\}$ that depends on the initial sample $X$, we proceed as follows.
We fix an angle $\theta$ and,
for each $A_i$, we construct a rotation $R_i$ that rotates $A_i$ about an axis orthogonal to $A_i$.
Specifically, if $A_i = (a_{1i}, a_{2i}, a_{3i})$ and $a_{3i} \neq 0$, $R_i$ rotates $A_i$ around the axis $\ell_i = (-a_{2i}, a_{1i}, 0)$ by the angle $\theta$.
If $a_{3i} = 0$, $R_i$ rotates $A_i$ around $(0, 0, 1)$ by $\theta$.
Next, we introduce Gaussian noise to the rotated sample
\begin{equation*}
A'_i = R_iA_i + W_i,
\end{equation*}
where $W_i \sim N_3(0,\varepsilon I_{3})$, for some specified level of randomness $\varepsilon$.

By varying the angle $\theta$, we control the value of Rcorr.
If $\theta = 0$, we expect the Rcorr between the datasets $X$ and $Y$ to be 1, and $-1$ if $\theta = \pi$.
To control the Fréchet mean of the sample $Y$, we select a matrix $B$ in $SO(3)$ such that $d(B, I_3) < \pi/2$, ensuring that the condition in Proposition \ref{Rcorr at midpoint of Frechet sample menas convergence} is met, which requires the Fréchet means to lie within a convex set.

Using \( \Phi \) and the exponential map at the identity, we declare the dependent sample $Y$ to be:
\begin{equation*}
Y_i = B\exp(\Phi(A'_i)), 
\end{equation*}
for $i = 1, \ldots, N$. Because left multiplication is an isometry, the Fréchet sample mean of $Y_1, \ldots, Y_N$ converges to $B$ as $N\to\infty$.
Similar to the simulation for $\mathbb{S}^2$-data, the scheme comprises four different scenarios, which vary based on the type of dependence and whether the Fréchet means are common or distinct.
The sample Riemannian correlation is evaluated according to Propositions \ref{Rcorr at Frechet sample mean convergence} and \ref{Rcorr at midpoint of Frechet sample menas convergence} for the cases of common means and distinct means, respectively.

\subsubsection{Simulation results}

Figure \ref{fig:Rcorr_vs_dcorr_samemean_zero_rotation} shows the two measures of correlation for dependent datasets with the same Fréchet mean, $B=I_3$ and varying noise level.
For small noise levels, the two measures equally capture the stochastic dependence. However, as the noise increases, the sample Rcorr declines faster than the empirical dcorr. Indeed, the sample Rcorr gets close to zero for higher noise levels, while the empirical dcorr still shows a positive value.

On the other hand, Figure \ref{fig:Rcorr_vs_dcorr_differentmeans_rotpi6} compares the correlation measures for two dependent datasets with different Fréchet means, $B \neq I_3$.
Similar to the previous case, the rapid decline of Rcorr is evident compared to dcorr.
However, in contrast to what we observed for the case of $\mathbb{S}^2$ previously, at zero noise, there is a gap between the empirical dcorr and the sample Rcorr, i.e. the sample Rcorr in not able to fully capture the strength of the dependence.

In Figure \ref{fig:Rcorr_vs_dcorr_samemean_rotpi_negative_Rcorr}, the parameters are chosen so that the two datasets are oppositely correlated at their common Fréchet mean.
The sample Rcorr again correctly captures the direction of this correlation, whereas dcorr captures the association in absolute value.

Finally, Figure \ref{fig:Rcorr_for_independent_datasets_in_SO(3)} shows the correlation measures for two independent datasets. As in the case of $\mathbb{S}^2$, the sample Rcorr captures independence more effectively than empirical dcorr for independent datasets.
The simulations for $SO(3)$ confirm the effectiveness of Rcorr in capturing both the direction and strength of dependence between $SO(3)$ datasets.

\begin{figure}[H]
    \centering
    \begin{subfigure}[t]{0.45\textwidth}
        \centering
        \includegraphics[width=\textwidth]{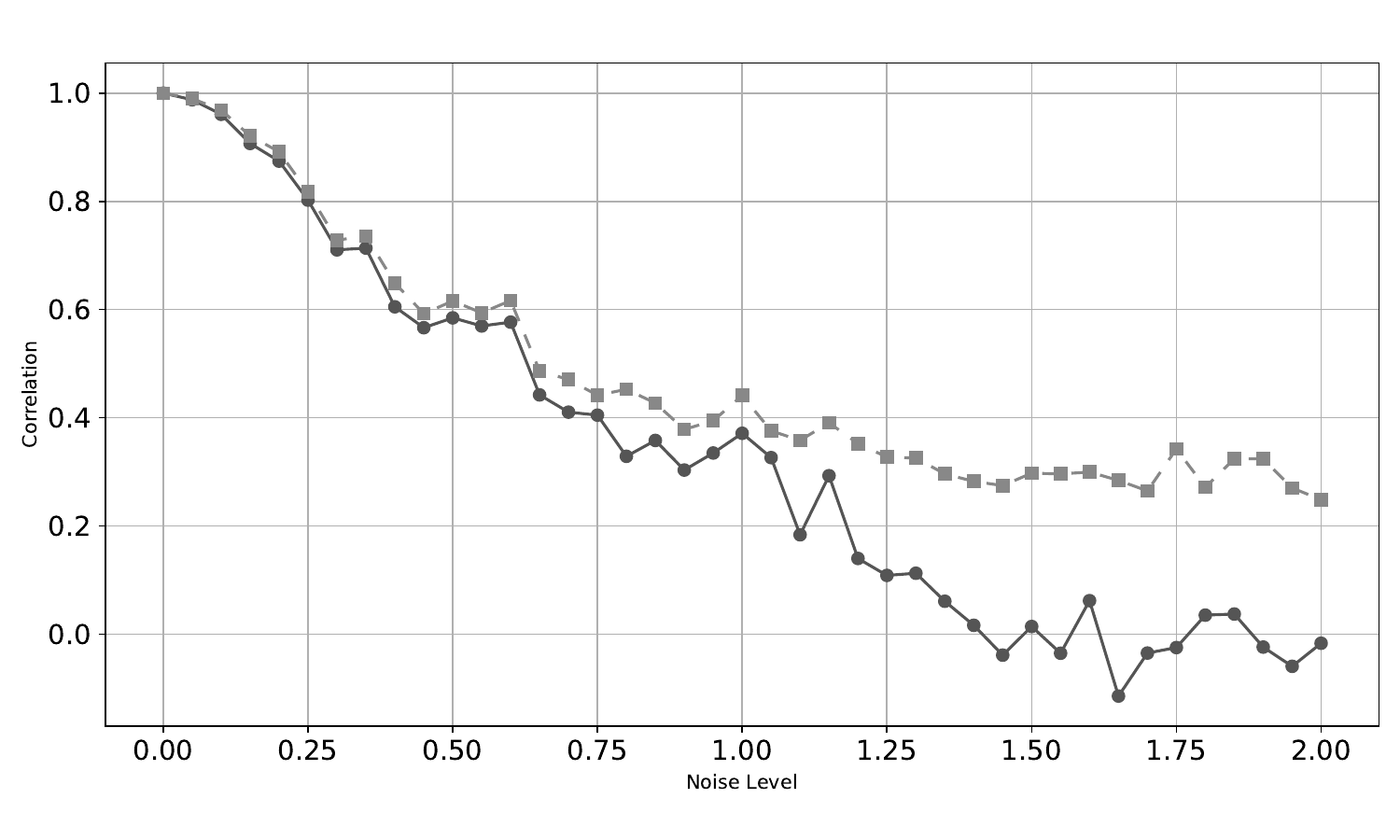}
        \caption{Comparison of Rcorr and dcorr against noise level for datasets in $SO(3)$ with the same Fréchet means. The parameters for the generative process are $\alpha=0.6$, sample size $=100$, and no rotation.}
        \label{fig:Rcorr_vs_dcorr_samemean_zero_rotation}
    \end{subfigure}
    \hfill
    \begin{subfigure}[t]{0.45\textwidth}
        \centering
        \includegraphics[width=\textwidth]{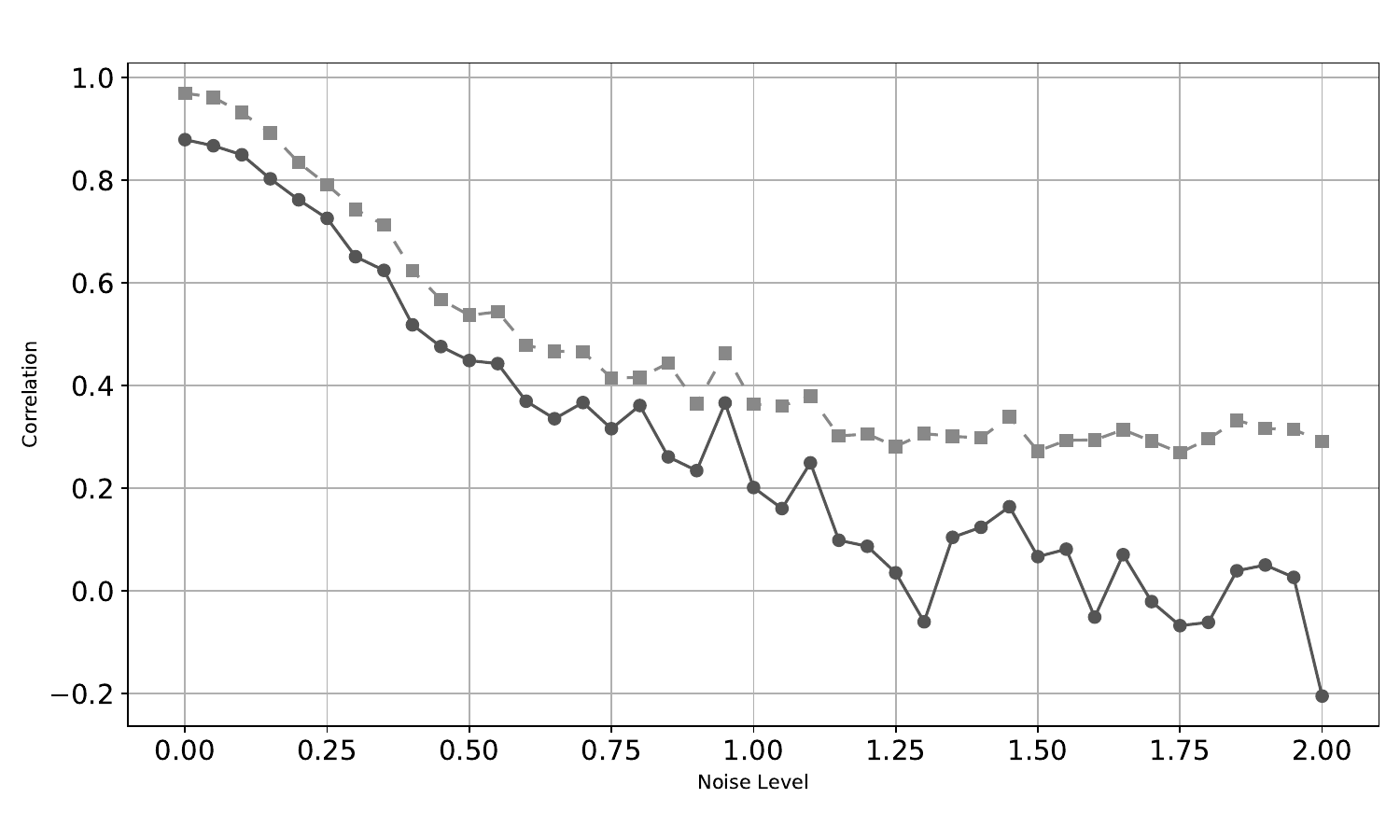}
        \caption{Comparison of Rcorr and dcorr against noise level for datasets in $SO(3)$ with different Fréchet means. The parameters for the generative process are $\alpha=0.6$, sample size $=100$, $B=\exp\begin{pmatrix}
0 & 1 & 0 \\
-1 & 0 & 0 \\
0 & 0 & 0
\end{pmatrix}$, and angle of rotation $=\pi/6$.}
        \label{fig:Rcorr_vs_dcorr_differentmeans_rotpi6}
    \end{subfigure}
    \vskip\baselineskip
    \begin{subfigure}[t]{0.45\textwidth}
        \centering
        \includegraphics[width=\textwidth]{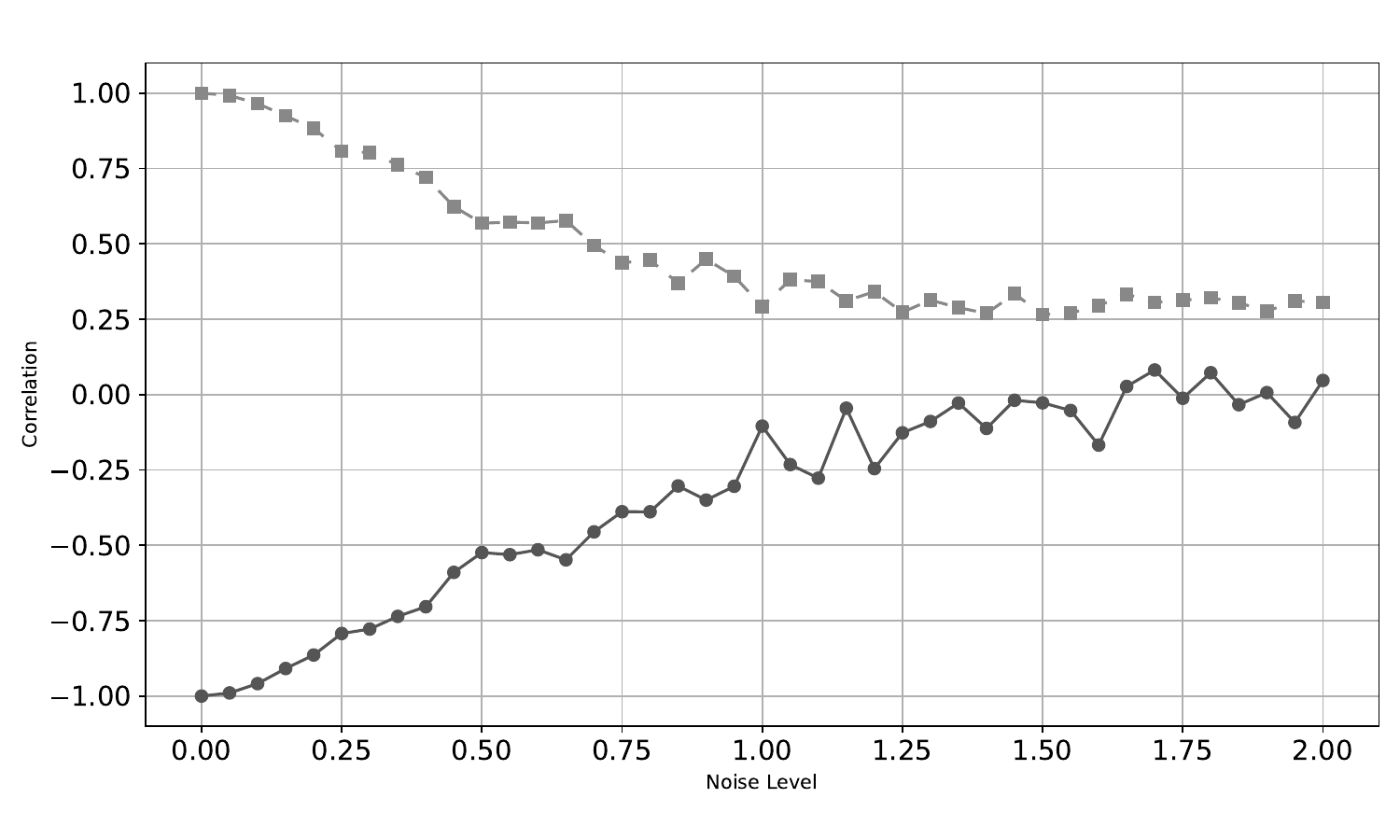}
        \caption{Comparison of Rcorr and dcorr against noise level for datasets in $SO(3)$ with the same Fréchet means. The parameters for the generative process are $\alpha=0.6$, sample size $=100$, axis of rotation $=(1,1,0)$, and angle of rotation $=\pi$.}
        \label{fig:Rcorr_vs_dcorr_samemean_rotpi_negative_Rcorr}
    \end{subfigure}
    \hfill
    \begin{subfigure}[t]{0.45\textwidth}
        \centering
        \includegraphics[width=\textwidth]{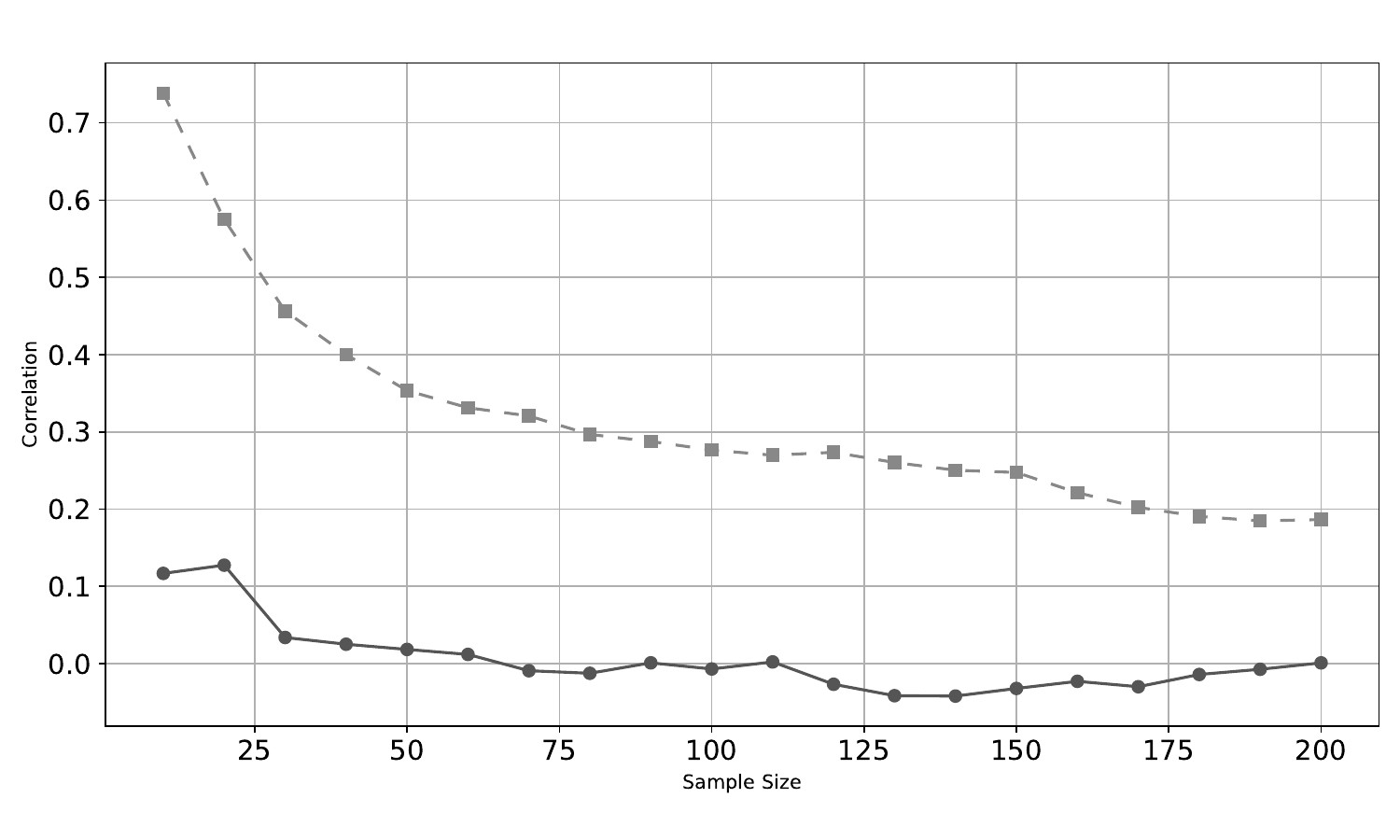}
        \caption{Comparison of Rcorr and dcorr for independent datasets in $SO(3)$ with the same Fréchet means and increasing sample sizes. For both datasets $\alpha=0.6$.}
        \label{fig:Rcorr_for_independent_datasets_in_SO(3)}
    \end{subfigure}
    \caption{Comparison of Rcorr (solid line with circular makers) and dcorr (dashed line with square markers) for datasets in $SO(3)$ under different scenarios.}
    \label{fig:SO(3)_comparison_all}
\end{figure}

\subsection{Application to vectorcardiogram data}
We now illustrate the behavior of the sample Riemannian correlation with an application to data in cardiology.
The vectorcardiogram (VCG) is a technique used in cardiology to measure the electrical activity of the heart. Unlike the electrocardiogram (ECG), the VCG provides data in the form of vectors in $\mathbb{R}^3$, offering a three-dimensional representation of the heart's electrical impulses. This method offers several significant advantages over the ECG, making it a valuable tool in clinical practice \cite{riera2007significance}.

In particular, the VCG measures the net electromotive force generated during the depolarization process of the ventricles. At the end of each depolarization cycle, VCG recordings represent an oriented loop in $\mathbb{R}^3$ known as the QRS loop. 
The points on the QRS loop correspond to vectors representing the resultant vector of the heart's electrical activity at each instant during the cardiac cycle.
Of special interest for clinicians is the direction of the vector with the largest magnitude, in the Euclidean norm.
From a mathematical perspective, these directional data can be considered spherical data, i.e., points on the unit sphere, see \cite{downs2003spherical} and \cite{riera2007significance}.
Two commonly used systems for VCG measurement are the Frank system and the McFee-Parungao system, which we refer to as the F-system and MP-system, respectively.
Because each system employs a different approach to measuring the QRS loop, the resulting data differ.
We consider the question of correlation between the results obtained from these two systems.
For more on the designs of these two systems and others, we refer to \cite{malmivuo1995bioelectromagnetism}.

A clinical dataset, consisting of directions of the maximum vectors in QRS loops, was taken for 25 girls using both systems.
The datasets can be found in Table 1 of \cite{downs2003spherical}. Figure \ref{fig:plot_F_system_MP_system} represents these data as points on the sphere. Evidently, the measurements from the two systems have different Fréchet means.

We want to examine the correlation between the measurements from the two systems.
To do that, we compute the midpoint sample Riemannian correlation of the two measurements, alongside the empirical distance correlation.
The results can be found in Table \ref{tab:correlation_measures}.

\begin{table}[htbp]
\centering
\begin{tabularx}{0.6\textwidth}{Xr}
    \toprule
    \textbf{Correlation Measure} & \textbf{Value} \\
    \midrule
    Empirical distance Correlation (dcorr) & 0.77086 \\
    Sample Riemannian Correlation at Midpoint (Rcorr) & 0.76777 \\
    \bottomrule
\end{tabularx}
\caption{Results for the VCG dataset from \cite{downs2003spherical}}
\label{tab:correlation_measures}
\end{table}

The empirical distance correlation suggests a relatively strong dependence/ association between the measurements of the two systems.
The sample Riemannian correlation at the midpoint is in agreement about the strength of the association.
Since the value of the sample Riemannian correlation is very close to one of the empirical distance correlation, we can infer that the association between the F-system and MP-system is explained by Rcorr at the midpoint.
Specifically, the measurements of the two systems tend to move in the same direction along the geodesics originating from the midpoint of their respective Fréchet means.
This gives us insights of the nature of the dependence between the measurements of the two systems.

These observations are consistent with the findings in Example 1 of \cite{downs2003spherical}, where methods for spherical regression data were applied.

\begin{figure}[H]
\centering
\includegraphics[width=0.7\textwidth]{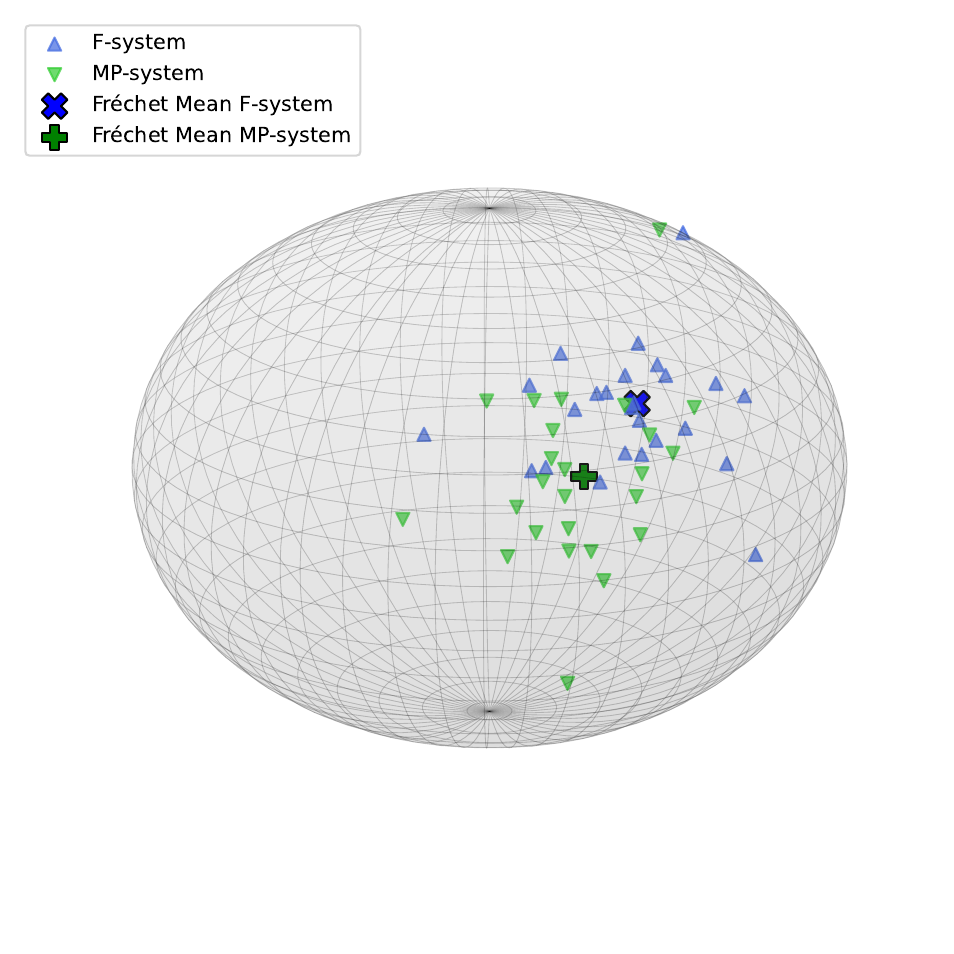}
\caption{Plot of the F-system and the MP-system datasets.}
\label{fig:plot_F_system_MP_system}
\end{figure}

\section{Conclusions and Future directions}
We introduced a novel approach to measure stochastic dependence between manifold-valued random variables, based on local measures of dependence that generalise the concepts of covariance and correlation for Euclidean-valued random variables.
These measures utilise the established framework of Fréchet moments for manifold-valued random variables, providing an intrinsic framework for analysing data in Riemannian manifolds.

We also provided consistent estimators for the Riemannian covariance and the Riemannian correlation and we demonstrated the effectiveness of these measures through simulation studies and real-world datasets. 
The simulation results showed that the proposed estimators capture the dependence (of lack thereof) for manifold-valued data and perform well compared to the existing measure of dependence based on distance correlation.

An interesting open question is how the Riemannian covariance depends on the choice of the reference point $p$. We suggested a practical approach to choose $p$ based on its centrality with respect to the data clouds, so that the resulting tangent space can provide a good approximation for the data. On the other hand, it is possible to imagine a more principled approach for the choice of the reference point, where the Riemannian covariance/correlation is treated as a function of the reference point and the variation is explored, for example along the geodesic connecting the two means, using a similar approach to the one discussed at the end of Section \ref{Different means: Common tangent space section}. 
Alternatively, rather than focusing on the geodesic between the Fréchet means, we could look for maxima and minima of the Riemannian covariance/correlation across the manifold, since the range obtained may provide insights into the dependence between the two variables. This line of research will pose an optimization problem within the framework of Riemannian manifold, but also a question of interpretation on how to compare measures on different tangent spaces. This can also be useful in scenarios where the Fréchet means are not unique, and therefore the reference points we used cannot be uniquely identified.

\section*{Codes and reproducibility}
The codes used for the simulations in the paper, as well as those for processing the vectorcardiogram data, can be found in the GitHub repository: 
\url{https://github.com/MeshalAbuqrais/Rcorr-simulations}
\appendix
\section{Additional proofs}
\label{sec:appendix}
This appendix contains proofs for some of the statements in the article.

\subsection{Proof of Proposition \ref{proposition Rcov(X,X) and Frechet function simple relation}}
\begin{proof}
Because the logarithm is radial isometry, we have
\begin{equation*}
d(X,p)=\abs{\abs{\log_{p} X}}=\sqrt{(\log_{p}X)^T(\log_p X)}.
\end{equation*}
Using this fact, it immediately follows that 
\begin{align*}
tr(\Sigma_{p}(X,X))&=E[(\log_{p}X)^T(\log_p X)]-E\left[\log_{p}X\right]^TE\left[\log_{p}X\right]
\\
&=E[d^{2}(X,p)]-E\left[\log_{p}X\right]^TE\left[\log_{p}X\right]=\mathcal{F}_{2}(p)-\norm{E\left[\log_{p}X\right]}^2.
\end{align*}
If $p=\mu$, then by the characterising property of Fréchet mean, we have $E[\log_{\mu}X]=0$. Thus, Rcov$_{\mu}(X,X)=\mathcal{F}_{2}(\mu)$.
\end{proof}
\subsection{Proof of Propostion \ref{proposition: Rcorr is between -1 and 1}}
\begin{proof}
First note that $\text{Rcorr}_p(X,Y)$ is related to $\text{Rcov}_{p}(X,Y)$ by
\begin{equation}
\text{Rcorr}_{p}(X,Y)=tr(\mathcal{R}_p(X,Y))=\frac{\text{Rcov}_{p}(X,Y)}{\sqrt{tr(\Sigma_{p}(X,X))} \sqrt{tr(\Sigma_{p}(Y,Y))}}.
\end{equation}
and $\text{Rcov}_{p}(X,Y)$ can be written as
\begin{equation}
\text{Rcov}_{p}(X,Y)=E\left[\left(\log_{p}X-E\left[\log_{p}X\right]\right)^T\left(\log_{p}Y-E\left[\log_{p}Y\right]\right)\right].
\end{equation}
We have
\begin{align*}
\abs{\text{Rcov}_{p}(X,Y)}&=\abs{E\left[\left(\log_{p}X-E\left[\log_{p}X\right]\right)^T\left(\log_{p}Y-E\left[\log_{p}Y\right]\right)\right]} \\
&\leq E\left[\abs{\left(\log_{p}X-E\left[\log_{p}X\right]\right)^T\left(\log_{p}Y-E\left[\log_{p}Y\right]\right)}\right].
\end{align*}
By the Cauchy-Schwarz inequality,
\begin{align*}
&E\left[\abs{\left(\log_{p}X-E\left[\log_{p}X\right]\right)^T\left(\log_{p}Y-E\left[\log_{p}Y\right]\right)}\right] \\
&\leq E\left(\norm{\log_{p}X-E\left[\log_{p}X\right]}^2\right)^{1/2} E\left(\norm{\log_{p}Y-E\left[\log_{p}Y\right]}^2\right)^{1/2} \\
&=E\left[tr\left(\left(\log_{p}X-E\log_{p}X\right)^T\left(\log_{p}X-E\log_{p}X\right)\right)\right]^{1/2} \\
&\quad \times E\left[tr\left(\left(\log_{p}Y-E\log_{p}Y\right)^T\left(\log_{p}Y-E\log_{p}Y\right)\right)\right]^{1/2}.
\end{align*}
Therefore, by interchanging the trace and the expectations
\begin{equation}
\abs{Rcov_{p}(X,Y)}^2 \leq tr\left(\Sigma_{p}(X,X)\right) tr\left(\Sigma_{p}(Y,Y)\right)
\end{equation}
which gives $\abs{\text{Rcorr}(X,Y)_p}\leq 1$.
\end{proof}

\subsection{Proof of Lemma \ref{lemma: boundedness of Frechet function}}
\begin{proof}
First we show that $\mathcal{F}_{r,X}$ is continuous. Let $p,q\in M$ and $\varepsilon>0$.
\begin{equation*}
\abs{\mathcal{F}_{r,X}(p)-\mathcal{F}_{r,X}(q)}\leq \int_{M}\abs{d^r{p,x}-d^{r}(q,x)}dQ_{X}(x)\leq \text{Vol}(M)\sup_{x\in M}\abs{d^{r}(p,x)-d^{r}(q,x)}.
\end{equation*}
By continuity of the distance function $d$, it follows that if $\abs{\mathcal{F}_{r,X}(p)-\mathcal{F}_{r,X}(q)}< \varepsilon$, there exists $\delta>0$ such that 
\begin{equation*}
\abs{d^{r}(p,x)-d^{r}(q,x)}<\frac{\varepsilon}{\text{Vol}(M)},
\end{equation*}
whenever $d(p,q)<\delta$. Thus, $\mathcal{F}_{r,X}$ is continuous for all $r>0$.
Since $M$ is a compact metric space, the Fréchet function $\mathcal{F}_{r,X}$ attains both a maximum and a minimum.
\end{proof}
\subsection{Proof of Lemma \ref{lemma on almost sure convergence of log map}}
\begin{proof}
Let $p_n$ be a sequence as assumed in the lemma. Consider the geodesic ball $\mathcal{B}_{p_n}(r_n)$ centred at $p_n$ with radius
\begin{equation*}
r_{n}=\text{inj}(p)-d(p,p_n).
\end{equation*}
Since $p_n\overset{a.s.}{\to} p$,  $d(p,p_n)\to 0$ almost surely. Therefore, there exists $n_0$ such that for all $n>n_0$, $r<r_{n}<\text{inj}(p)$ and, thus,
\begin{equation*}
\log_{p_n}X,
\end{equation*}
is well-defined. By continuity of the logarithm map with respect to the base point we have
\begin{equation*}
\log_{p_n}X\overset{a.s.}{\longrightarrow} \log_{p}X,
\end{equation*}
as $n\to\infty$.
\end{proof}

\bibliographystyle{plainnat}
\bibliography{references}

\begin{thebibliography}{29}
\providecommand{\natexlab}[1]{#1}
\providecommand{\url}[1]{\texttt{#1}}
\expandafter\ifx\csname urlstyle\endcsname\relax
  \providecommand{\doi}[1]{doi: #1}\else
  \providecommand{\doi}{doi: \begingroup \urlstyle{rm}\Url}\fi

\bibitem[Afsari(2011)]{afsari2011riemannian}
Bijan Afsari.
\newblock Riemannian ${L}^{p}$ center of mass: existence, uniqueness, and
  convexity.
\newblock \emph{Proceedings of the American Mathematical Society}, 139\penalty0
  (2):\penalty0 655--673, 2011.

\bibitem[Aitchison(1982)]{aitchison1982statistical}
John Aitchison.
\newblock The statistical analysis of compositional data.
\newblock \emph{Journal of the Royal Statistical Society: Series B
  (Methodological)}, 44\penalty0 (2):\penalty0 139--160, 1982.

\bibitem[Berger(2003)]{berger2003panoramic}
Marcel Berger.
\newblock \emph{A panoramic view of Riemannian geometry}.
\newblock Springer, 2003.

\bibitem[Bharath et~al.(2018)Bharath, Kurtek, Rao, and
  Baladandayuthapani]{bharath2018radiologic}
Karthik Bharath, Sebastian Kurtek, Arvind Rao, and Veerabhadran
  Baladandayuthapani.
\newblock Radiologic image-based statistical shape analysis of brain tumours.
\newblock \emph{Journal of the Royal Statistical Society Series C: Applied
  Statistics}, 67\penalty0 (5):\penalty0 1357--1378, 2018.

\bibitem[Buser and Karcher(1981)]{buser1981karcher-gromov}
P.~Buser and H.~Karcher.
\newblock \emph{Gromov's Almost Flat Manifolds}.
\newblock Asterisque : No. 81, 1981. Soci{\'e}t{\'e} math{\'e}matique de
  France, 1981.
\newblock URL \url{https://books.google.co.uk/books?id=ZhIZAQAAIAAJ}.

\bibitem[Downs(2003)]{downs2003spherical}
TD~Downs.
\newblock Spherical regression.
\newblock \emph{Biometrika}, 90\penalty0 (3):\penalty0 655--668, 2003.

\bibitem[Dryden and Mardia(2016)]{dryden2016statistical}
Ian~L Dryden and Kanti~V Mardia.
\newblock \emph{Statistical shape analysis: with applications in R}.
\newblock John Wiley \& Sons, 2016.

\bibitem[Fletcher and Joshi(2007)]{fletcher2007riemannian}
P~Thomas Fletcher and Sarang Joshi.
\newblock Riemannian geometry for the statistical analysis of diffusion tensor
  data.
\newblock \emph{Signal Processing}, 87\penalty0 (2):\penalty0 250--262, 2007.

\bibitem[Fr{\'e}chet(1948)]{frechet1948elements}
Maurice Fr{\'e}chet.
\newblock Les {\'e}l{\'e}ments al{\'e}atoires de nature quelconque dans un
  espace distanci{\'e}.
\newblock In \emph{Annales de l'institut Henri Poincar{\'e}}, volume~10, pages
  215--310, 1948.

\bibitem[Hanson and Cunningham(2006)]{hanson2006visualizing}
A.J. Hanson and S.~Cunningham.
\newblock \emph{Visualizing Quaternions}.
\newblock The Morgan Kaufmann Series in Interactive 3D Technology. Elsevier
  Science, 2006.
\newblock ISBN 9780080474779.
\newblock URL \url{https://books.google.co.uk/books?id=CoUB09xzme4C}.

\bibitem[Hjorth et~al.(2002)Hjorth, Kokkendorff, and
  Markvorsen]{hjorth2002hyperbolic}
P~Hjorth, S~Kokkendorff, and Steen Markvorsen.
\newblock Hyperbolic spaces are of strictly negative type.
\newblock \emph{Proceedings of the American Mathematical Society}, 130\penalty0
  (1):\penalty0 175--181, 2002.

\bibitem[Jakob(2012)]{jakob2012numerically}
Wenzel Jakob.
\newblock Numerically stable sampling of the von mises-fisher distribution on
  $\mathbb{S}^2$ (and other tricks).
\newblock \emph{Interactive Geometry Lab, ETH Z{\"u}rich, Tech. Rep}, 6, 2012.

\bibitem[Lee(2018)]{Rie-Lee}
John~M Lee.
\newblock \emph{Introduction to Riemannian manifolds}, volume~2.
\newblock Springer, 2018.

\bibitem[Lyons(2013)]{distance-covariance-in-metric-spaces-Lyons}
Russell Lyons.
\newblock {Distance covariance in metric spaces}.
\newblock \emph{The Annals of Probability}, 41\penalty0 (5):\penalty0 3284 --
  3305, 2013.
\newblock \doi{10.1214/12-AOP803}.
\newblock URL \url{https://doi.org/10.1214/12-AOP803}.

\bibitem[Malmivuo and Plonsey(1995)]{malmivuo1995bioelectromagnetism}
Jaakko Malmivuo and Robert Plonsey.
\newblock \emph{Bioelectromagnetism: principles and applications of bioelectric
  and biomagnetic fields}.
\newblock Oxford University Press, USA, 1995.

\bibitem[Mardia and Jupp(2009)]{mardia2009directional}
Kanti~V Mardia and Peter~E Jupp.
\newblock \emph{Directional statistics}.
\newblock John Wiley \& Sons, 2009.

\bibitem[Marron and Alonso(2014)]{marron2014overview}
J~Steve Marron and Andr{\'e}s~M Alonso.
\newblock Overview of object oriented data analysis.
\newblock \emph{Biometrical Journal}, 56\penalty0 (5):\penalty0 732--753, 2014.

\bibitem[Marron and Dryden(2021)]{marron2021object}
James~Stephen Marron and Ian~L Dryden.
\newblock \emph{Object oriented data analysis}.
\newblock Chapman and Hall/CRC, 2021.

\bibitem[Menafoglio et~al.(2021)Menafoglio, Guadagnini, Guadagnini, and
  Secchi]{menafoglio2021object}
Alessandra Menafoglio, Laura Guadagnini, Alberto Guadagnini, and Piercesare
  Secchi.
\newblock Object oriented spatial analysis of natural concentration levels of
  chemical species in regional-scale aquifers.
\newblock \emph{Spatial Statistics}, 43:\penalty0 100494, 2021.

\bibitem[Pan et~al.(2020)Pan, Wang, Zhang, Zhu, and Zhu]{pan2020ball}
Wenliang Pan, Xueqin Wang, Heping Zhang, Hongtu Zhu, and Jin Zhu.
\newblock Ball covariance: A generic measure of dependence in banach space.
\newblock \emph{Journal of the American Statistical Association}, 2020.

\bibitem[Patrangenaru and Ellingson(2016)]{patrangenaru2016nonparametric}
Victor Patrangenaru and Leif Ellingson.
\newblock \emph{Nonparametric statistics on manifolds and their applications to
  object data analysis}.
\newblock CRC Press, Taylor \& Francis Group Boca Raton, 2016.

\bibitem[Pennec(2006)]{pennec2006intrinsic}
Xavier Pennec.
\newblock Intrinsic statistics on riemannian manifolds: Basic tools for
  geometric measurements.
\newblock \emph{Journal of Mathematical Imaging and Vision}, 25:\penalty0
  127--154, 2006.

\bibitem[Pennec et~al.(2019)Pennec, Sommer, and
  Fletcher]{pennec-sommer-fletcher-2019riemannian-medical}
Xavier Pennec, Stefan Sommer, and Tom Fletcher.
\newblock \emph{Riemannian geometric statistics in medical image analysis}.
\newblock Academic Press, 2019.

\bibitem[Riera et~al.(2007)Riera, Uchida, Ferreira~Filho, Meneghini, Ferreira,
  Schapacknik, Dubner, and Moffa]{riera2007significance}
Andr{\'e}s Ricardo~P{\'e}rez Riera, Augusto~H Uchida, Celso Ferreira~Filho,
  Adriano Meneghini, Celso Ferreira, Edgardo Schapacknik, Sergio Dubner, and
  Paulo Moffa.
\newblock Significance of vectorcardiogram in the cardiological diagnosis of
  the 21st century.
\newblock \emph{Clinical cardiology}, 30\penalty0 (7):\penalty0 319, 2007.

\bibitem[Severn et~al.(2022)Severn, Dryden, and Preston]{severn2022manifold}
Katie~E Severn, Ian~L Dryden, and Simon~P Preston.
\newblock Manifold valued data analysis of samples of networks, with
  applications in corpus linguistics.
\newblock \emph{The Annals of Applied Statistics}, 16\penalty0 (1):\penalty0
  368--390, 2022.

\bibitem[Shao et~al.(2022)Shao, Lin, and Yao]{shao2022intrinsic}
Lingxuan Shao, Zhenhua Lin, and Fang Yao.
\newblock Intrinsic riemannian functional data analysis for sparse longitudinal
  observations.
\newblock \emph{The Annals of Statistics}, 50\penalty0 (3):\penalty0
  1696--1721, 2022.

\bibitem[Sz{\'e}kely et~al.(2007)Sz{\'e}kely, Rizzo, and
  Bakirov]{szekely2007measuring}
G{\'a}bor~J. Sz{\'e}kely, Maria~L. Rizzo, and Nail~K. Bakirov.
\newblock {Measuring and testing dependence by correlation of distances}.
\newblock \emph{The Annals of Statistics}, 35\penalty0 (6):\penalty0 2769 --
  2794, 2007.
\newblock \doi{10.1214/009053607000000505}.
\newblock URL \url{https://doi.org/10.1214/009053607000000505}.

\bibitem[Tu(2017)]{tu2017differential}
Loring~W Tu.
\newblock \emph{Differential geometry: connections, curvature, and
  characteristic classes}, volume 275.
\newblock Springer, 2017.

\bibitem[Zhan et~al.(2019)Zhan, Ma, Liu, and Shimizu]{zhan2019circular}
Xiaoping Zhan, Tiefeng Ma, Shuangzhe Liu, and Kunio Shimizu.
\newblock On circular correlation for data on the torus.
\newblock \emph{Statistical papers}, 60:\penalty0 1827--1847, 2019.

\end{thebibliography}

\end{document}